\documentclass[12pt]{amsart}
\usepackage[lite]{amsrefs}

\usepackage{amssymb}
\usepackage[T1]{fontenc}
\usepackage{mathrsfs}  
\usepackage[all,cmtip]{xy}

\usepackage {amsfonts, amssymb, amscd,amsthm, amsmath 
}

\usepackage{amsbsy}
\usepackage[all]{xy}

\usepackage{color}

\theoremstyle{plain}
{
  \newtheorem{thm}{Theorem}[section]
  \newtheorem{defn}[thm]{Definition}
  \newtheorem{cor}[thm]{Corollary}
  \newtheorem{lem}[thm]{Lemma}
  \newtheorem{prop}[thm]{Proposition}
  \theoremstyle{definition}
  \newtheorem{rem}[thm]{Remark}
  \theoremstyle{plain}
  \newtheorem{clm}[thm]{Claim}
  \newtheorem{notation}[thm]{Notation}
  
  \newtheorem{constr}[thm]{Construction}
  \newtheorem{sssection}[thm]{}
}

{

}
\renewcommand{\subsubsection}{\sssection\rm}

\newcommand{\can}{\text{\rm can}}

\newcommand{\id}{\text{\rm id}}
\newcommand{\pr}{\text{\rm pr}}

\newcommand{\red}{\text{\rm red}}
\newcommand{\inc}{\text{\rm inc}}

\newcommand{\const}{\text{\rm const}}
\newcommand{\Spec}{\text{\rm Spec}}

\newcommand{\Ker}{\text{\rm Ker}}

\newcommand{\Aff}{\mathbf {A}}
\newcommand{\Pro}{\mathbf {P}}

\newcommand{\cal}{\mathcal}

\newcommand \xra {\xrightarrow }

\newcommand \hra {\hookrightarrow }

\newcommand{\ttf}{{\text{f}}}

\renewcommand{\P}{\mathbb P}

\newcommand\mydim{\text{\rm dim}}

\renewcommand \id{\operatorname{id}}

\renewcommand \phi\varphi

\newcommand{\ad}{ad}

\newcommand{\R}{{\rm R}}

\begin{document}

\title
{Weak elementary fibrations}

\author{Ning GUO}
\author{Ivan Panin}
\address{St. Petersburg branch of V. A. Steklov Mathematical Institute, Fontanka 27, 191023 St. Petersburg, Russia}
\email{guo.ning@eimi.ru}
\email{paniniv@gmail.com}
\date{\today}

\maketitle

\begin{abstract}
We introduce a notion of a weak elementary fibration and prove that it does exist in certain interesting cases.
Our notion is a modification of the M. Artin's notion of an elementary fibration.
\end{abstract}

\section{Weak elementary fibration: Definition and Result}\label{s:w_el_fibr_inf field}
In this Section we modify a result of M. Artin from~\cite{LNM305} concerning existence of nice neighborhoods.
The following notion is a modification of the one introduced by Artin in~\cite[Exp.~XI, D\'ef.~3.1]{LNM305}
and of its version introduced in \cite[Defn.~2.1]{PSV}.
In this preprint the term {\bf scheme} means a separated, quasi-compact, Noetherian of finite Krull dimension scheme
(as in \cite{MV}).
{\bf The main result} of the preprint is Theorem \ref{Very_General} stated right below.
This Theorem is subdivided into two cases: 
\begin{itemize}
\item[1)] the residue field $k(b)$ at the closed point $b\in B$ is finite (Theorem \ref{cor: w_el_fib_fin_A});
\item[2)] the residue field $k(b)$ at the closed point $b\in B$ is infinite (Theorem \ref{thm_Artin_2_A}).
\end{itemize}

\begin{defn}
\label{DefnElemFib}
Let $T$ be a scheme. A scheme morphism $q: T\to S$ is called
a \emph{weak elementary fibration} if this morphism can be included in a commutative diagram
\begin{equation}
\label{SquareDiagram}
    \xymatrix{
     T\ar[drr]_{q}\ar[rr]^{j}&&
\bar T\ar[d]_{\bar q}&&T_{\infty}\ar[ll]_{i}\ar[lld]^{q_{\infty}} &\\
     && S  &\\    }
\end{equation}
of morphisms satisfying the following conditions:
\begin{itemize}
\item[{\rm(i)}] $j$ is an open immersion dense at each fibre of
$\bar q$, and $T=\bar X-T_{\infty}$;
\item[{\rm(ii)}]
$\bar q$ is smooth projective all of whose fibres are equi-dimensional of dimension one;
\item[{\rm(iii)}] $q_{\infty}$ is a finite \'{e}tale morphism (all of whose fibres are non-empty);
\item[{\rm(iv)}] $q=pr_S \circ \pi$, where $\pi: T\to \mathbb A^1_S$ is a finite flat surjective morphism.
\end{itemize}
The diagram \eqref{SquareDiagram} is called \emph{a diagram of the weak elementary fibration} $q: T\to S$.

Let $Z\subset T$ be a closed subscheme.
A scheme morphism $q: T\to S$ is called
a \emph{weak elementary $Z$-fibration}
if it is a weak elementary fibration and the morphism
$q|_Z: Z\to S$ is finite. In this case the diagram \eqref{SquareDiagram}
is called \emph{a diagram of the weak elementary $Z$-fibration} $q: T\to S$.
\end{defn}

\begin{notation}
Let $B$ be a local regular scheme and $b\in B$ its closed point.
For a $B$-scheme $Y$ put $Y_b=Y\times_B b$. For a $B$-morphism
$\phi: Y'\to Y''$ of $B$-schemes write $\phi_b$ for the $b$-morphism
$\phi\times_B b: Y'\times_B b\to Y''\times_B b$.
\end{notation}

\begin{thm}\label{Very_General}
Let $B$ be a local regular scheme and $b\in B$ its closed point.
Let $X/B$ be an irreducible $B$-smooth, $B$-projetive scheme of relative dimension $r$.
Let $\underbar x\subset X$ be a finite family of closed points. Then there is an open subset $S$ in $\mathbb A^{r-1}_B$,
an open subscheme $X_S\subset X$ containing $\underbar x$ and a weak elementary fibration
$q_S: X_S\to S$.

Let $Z\subset X$ be a closed subscheme such that $\dim Z_b <r$. Then there is an open subset $S$ in $\mathbb A^{r-1}_B$,
an open subscheme $X_S\subset X$ containing $\underbar x$ and a weak elementary $Z_S$-fibration
$q_S: X_S\to S$, where $Z_S=Z\cap X_S$.
\end{thm}

\begin{thm}\label{Very_General_Diagram}
Under the hypotheses of Theorem \ref{Very_General} there is an open
$S\subset \mathbb A^{r}_B$, an open neighborhood $\dot X_S$ of the set ${\bf x}$ in $X$ and
a weak elementary fibration $\dot q_S: \dot X_S\to S$ and a commutative diagram of $S$-schemes
\begin{equation}
\label{V_G_SquareDiagram2}
    \xymatrix{
     \dot X_S\ar[drr]_{\dot q_S}\ar[rr]^{j_S}&&
\hat X_S\ar[d]_{\hat q_S}&& \mathbb W_S\ar[ll]_{i}\ar[lld]^{pr_S} &\\
     && S  &\\    }
\end{equation}
which is a diagram of the weak elementary fibration $\dot q_S: \dot X_S\to S$.

Moreover, if $Z\subset X$ is as in Theorem \ref{Very_General}, then
the open $S\subset \mathbb A^{r}_B$, the open neighborhood $\dot X_S$ of the set ${\bf x}$ in $X$,
the weak elementary fibration $\dot q_S: \dot X_S\to S$ and
the commutative diagram \eqref{V_G_SquareDiagram2} of $S$-schemes
can be chosen such that the \eqref{V_G_SquareDiagram2}
is a diagram of the weak elementary $\dot Z_S$-fibration
$\dot q_S: \dot X_S\to S$.
\end{thm}
\begin{rem}\label{ext_of_V_G_SquareDiagram2}
It turns out that the diagram \eqref{V_G_SquareDiagram2} comes equipped with a diagram of $B$-schemes of the form
\begin{equation}
\label{V_G_Additional}
    \xymatrix{
     X                       &&  \hat X \ar[ll]^{\tau}  &  \\
     \dot X_S \ar[u]^{in} \ar[rr]^{j_S} && \hat X_S \ar[u]_{In} && \mathbb W_S \ar[ll]^{i_S} &\\
     \dot Z_S \ar[u]^{}                       && \tau^{-1}(\dot Z_S) \ar[ll]^{\cong \ \tau} \ar[u]^{} && \emptyset \ar[u]^{} \ar[ll]^{}&\\    }
\end{equation}
such that
\begin{itemize}
\item[{\rm(i)}] $in$, $In$ are open embedding and $in=\tau\circ In\circ j$;
\item[{\rm(ii)}]
the morphism $\dot X_S \xleftarrow{\tau|_{\hat X_S-\mathbb W_S}} \hat X_S-\mathbb W_S$ is well-defined and it is an isomorphism;
\item[{\rm(iii)}]
$(\tau|_{\hat X_S-\mathbb W_S})\circ j=id_{\dot X_S}$ and $j\circ (\tau|_{\hat X_S-\mathbb W_S})=id_{\hat X_S-\mathbb W_S}$;
\item[{\rm(iv)}]
the morphism $j: \dot X_S\to \hat X_S-\mathbb W_S$ is an isomorphism;
\item[{\rm(v)}]
$\tau^{-1}(\dot Z_S)\cap \mathbb W_S=\emptyset$;
\end{itemize}
\end{rem}
Let ${\bf G}$ be a reductive $B$-group scheme and $\cal E$ be a principal ${\bf G}$-bundle over $X$.
Under the hypotheses of Theorem \ref{Very_General} we easily get the following result as
an immediate consequence of Theorem \ref{Very_General_Diagram} and Remark \ref{ext_of_V_G_SquareDiagram2}
\begin{cor}\label{Bundle}
Let $Z\subset X$ be a closed subset as in Theorem \ref{Very_General}.
Consider the diagrams as in Theorem \ref{Very_General} and Remark \ref{ext_of_V_G_SquareDiagram2}.
Clearly,
$$\hat X_S=(\hat X_S-\mathbb W_S)\cup (\hat X_S-\tau^{-1}(\dot Z_S))\cong \dot X_S\cup (\hat X_S-\tau^{-1}(\dot Z_S))$$
Consider the
principal ${\bf G}$-bundle $\hat {\cal E}_S:= In^*(\tau^*(\cal E))$ over $\hat X_S$.
Suppose the bundle ${\cal E}|_{X-Z}$ is trivial.
Then following is true
\begin{itemize}
\item[{\rm(i)}]
the bundle $\hat {\cal E}_S$ being restricted to $\hat X_S-\tau^{-1}(\dot Z_S)$ is trivial;
\item[{\rm(ii)}]
the closed subscheme scheme $\tau^{-1}(\dot Z_S)$ of $\hat X_S$ is $S$-finite;
\item[{\rm(iii)}]
the bundle $j^*_S(\hat {\cal E}_S)={\cal E}|_{\dot X_S}$;
\item[{\rm(iv)}]
${\bf x}$ is in $\dot X_S$.
\end{itemize}
\end{cor}

\begin{rem}\label{Moral}
We begin with a principal ${\bf G}$-bundle $\cal E$ over $X$ trivial away of the closed subscheme $Z$
as in Theorem \ref{Very_General}. Then the {\it smooth projective curve} $\hat X_S$ over $S$,
the ${\bf G}$-bundle $\hat {\cal E}_S$ over $\hat X_S$, the set ${\bf x}$ and the closed subscheme $Z$
enjoy the following properties: \\
1) $(\hat {\cal E}_S)|_{\dot X_S}$ equals ${\cal E}|_{\dot X_S}$; \\
2) $(\hat {\cal E}_S)|_{\hat X_S-\tau^{-1}(\dot Z_S)}$ is trivial;\\
3) ${\bf x}$ is in $\dot X_S$;\\
4) the closed subscheme scheme $\tau^{-1}(\dot Z_S)$ of $\hat X_S$ is $S$-finite.
\end{rem}
\subsection*{Acknowledgements} 
The authors thank the excellent environment of the International Mathematical Center at POMI.
This work is supported by Ministry of Science and Higher Education of the Russian Federation, agreement \textnumero~ 075-15-2022-289. 
\section{A moving lemma }\label{sect:moving}
\begin{defn}\label{aff_rel_S}
Let $S$ be an affine scheme and $O_S=\Gamma(S,\cal O_S)$. Recall that an affine $S$-scheme is a morphism
$T\to S$ with an affine scheme $T$ such that $O_T$ is a finitely generated $O_S$-algebra.
\end{defn}
The following result is proved in \cite[Thm. 1.2]{GPan}.
\begin{thm}(\cite[Thm. 1.2]{GPan})\label{Moving_for_DVR}
Let $R$ be a discrete valuation ring, $X$ be an $R$-smooth $R$-projective irreducible $R$-scheme of pure
relative dimension $r$. Let ${\bf x}\subset X$ be a finite subset of closed points. Put
$U=Spec (\mathcal O_{X,{\bf x}})$.
Let $Z\subset X$ be a closed subset
avoiding all generic
points of the closed fibre of $X$.
Then there is
an affine $U$-smooth scheme $\cal X$ of relative dimension one
with a closed subset ${\cal Z}\subset {\cal X}$
fitting into the following commutative diagram
\begin{equation}
\label{diag:A1_homotopy}
    \xymatrix{
     \mathbb A^1_U - \tau(\cal Z) \ar[d]_{inc_{\P}}           &&  \cal X- \cal Z \ar[ll]^{\tau|_{\cal X- \cal Z}} \ar[d]_{Inc} \ar[rr]^{p_{X-Z}} && X-Z \ar[d]_{inc} &  \\
     \mathbb A^1_U                 && \cal X \ar[ll]^{\tau}  \ar[rr]^{p_X} && X  &\\
                          && \cal Z \ar[llu]^{\tau|_{\cal Z}} \ar[u]^{I} \ar[rr]^{p_Z} && Z \ar[u]^{i} &\\    }
\end{equation}
and subjecting the conditions listed below \\
(i) $\tau$ is an \'{e}tale morphism such that $\tau|_{\cal Z} : {\cal Z}\to \mathbb  A^1_U$ is a closed embedding;\\
(ii) the left square is a Nisnevich square in the category of smooth $U$-schemes;\\
(iii) both right-hand side squares are commutative; \\
(iv) the structure morphism $p_U : {\cal X}\to U$ defined as $pr_U\circ \tau$ permits a $U$-section
$\Delta: U \to \cal X$;\\
(v) we have $p_X \circ \Delta=can$, where $can: U \hra X$ is the open immersion;\\
(vi) the $i_0:=\tau \circ \Delta$ is the zero section $i_0: U \to \mathbb A^1_U$
of the projection $pr_U : \mathbb A^1_U\to U$; \\
(vii) $\cal Z$ is $U$-finite; \\
(viii) $\tau(\mathcal Z)\cap \{1\}\times U=\emptyset$
\end{thm}

\begin{rem}\label{Rem:Moving_for_DVR}
The closed subset
$Z\subset X$ avoids all generic
points of the closed fibre of $X$.
It follows that $\dim Z_b < r$.
Thus the $Z\subset X$ is as in Theorem \ref{Very_General}.
Put $B=\mathrm{Spec} (R)$ and let $b\in B$ is its the closed point.

By Theorem \ref{Very_General_Diagram}
an open
$S\subset \mathbb A^{r}_B$, an open neighborhood $\dot X_S$ of the set ${\bf x}$ in $X$
and a morphism $\dot q_S: \dot X_S\to S$
can be chosen such that the diagram \eqref{V_G_SquareDiagram2}
is a diagram of the weak elementary $\dot Z_S$-fibration
$\dot q_S: \dot X_S\to S$, where $\dot Z_S=Z\cap \dot X_S$.
Put $U=\mathrm{Spec} (\mathcal O_{X,{\bf x}})$.


We constructed
the right hand side squares of the diagram \eqref{V_G_SquareDiagram2}.
By Theorem \ref{Very_General_Diagram} they enjoy the properties (iii),(iv),(v) and (vii).


We still work in the infinite residue field case (the field $k(b)$ is infinite).
It remains to construct the left hand side square
and
to check
(i), (ii) and (vi).
In the infinite residue field case ($k(b)$ is infinite)
this is done in \cite[Section 4]{GPan}.

In the finite residue field case ($k(b)$ {\bf is finite})
sometimes just there is no morphism $\tau$ enjoying the propetry (i).
Say, this happen if $k(b)=\mathbb F_3$ and the number of $k(b)$-rational points
of ${\cal Z}_b$ is strictly greater that 3 (for instance 4 or 11).
In this case the diagram \eqref{diag:A1_homotopy} enjoying the propeties
(i) to (vii)
is constructed in
\cite[Section 3]{GPan}.
\end{rem}

\begin{thm}(\cite[Thm.~1.3]{GPan})
\label{MainHomotopy}
Let $R$, $X$, ${\bf x}\subset X$, $r$, $U=Spec (\mathcal O_{X,\bf x})$ be as in Theorem \ref{Moving_for_DVR}.
Let ${\bf G}$ be a reductive $R$-group scheme and
$\mathcal G/X$ be a principal ${\bf G}$-bundle
trivial over the generic point $\eta$ of $X$.
Then there
exists a principal ${\bf G}$-bundle $\mathcal G_t$ over $\mathbb A^1_U$
and a closed subset $\mathcal Z'$ in $\mathbb A^1_U$ finite over $U$
such that
\par
(i) the ${\bf G}$-bundle $\mathcal G_t$ is trivial over the open subscheme
$\mathbb A^1_U-\mathcal Z'$ in $\mathbb A^1_U$,
\par
(ii) the restriction of $\mathcal G_t$ to $\{0\}\times U$ coincides
with the restriction of $\mathcal G$ to $U$;
\par
(iii) $\mathcal Z'\cap \{1\}\times U=\emptyset$;
\end{thm}

\begin{proof}[Proof of Theorem \ref{MainHomotopy}]
The principal ${\bf G}$-bundle $\mathcal G$ is
trivial over the generic point $\eta$ of $X$.
The main theorem of \cite{NG} yields the triviality of the principal ${\bf G}$-bundle $\mathcal G|_U$.
In turn, this shows that there is a closed subset
$Z\subset X$ avoiding all generic
points of the closed fibre of $X$
and such that
the principal ${\bf G}$-bundle $\mathcal G|_{X-Z}$ is trivial.
Now consider the diagram \eqref{diag:A1_homotopy} as in
Theorem \ref{Moving_for_DVR}
and take the closed subscheme
${\cal Z}'=\tau(\cal Z)$
of the scheme $\mathbb A^1_U$.

Clearly, $\mathcal Z'$ is finite over $U$ and it enjoys the property (iii).
To construct a desired principal ${\bf G}$-bundle $\mathcal G_t$ over $\mathbb A^1_U$
note that the bundle $p^*_X({\cal E})|_{\cal X- \cal Z}$ is trivial. Since
the left square is a Nisnevich square in the category of smooth $U$-schemes
there is a principal ${\bf G}$-bundle $\mathcal G_t$ over $\mathbb A^1_U$
such that \\
1) $\tau^*(\mathcal G_t)=p^*_X({\cal E})$;\\
2) $\mathcal G_t|_{\mathbb A^1_U-\mathcal Z'}$ is trivial. \\
Now the properties (v) and (vi) of the diagram \eqref{diag:A1_homotopy}
show that
the principal ${\bf G}$-bundle $\mathcal G_t$
enjoys the property (ii).
That is the restriction of $\mathcal G_t$ to $\{0\}\times U$ coincides
with the restriction of $\mathcal G$ to $U$.

The proof of the theorem is completed.
\end{proof}

\section{A finite field case of Theorem \ref{Very_General}}\label{sect:f_field_case}
The main aim of this section is to state Theorem \ref{thm_Artin_2_fin},
which will be proven in section \ref{blowups}.
\begin{notation}\label{not:main}
Let $\mathbb F$ be a finite field, $n\geq 1$ an integer.\\
Let $\mathbb P^n:=\mathbb P^n_{\mathbb F}$ be the projective space over $\mathbb F$. \\
Let $S_{\mathbb F} = \mathbb F[T_0,...,T_n]$ be the homogeneous coordinate ring of $\mathbb P^n$.\\
let $S_{\mathbb F,d} \subset S_{\mathbb F}$ be the $\mathbb F$-subspace of homogeneous degree $d$ polynomials,\\
let $S_{\mathbb F,\mathrm{hom}} = \sqcup^{\infty}_{d=0} S_{\mathbb F,d}$ (the disjoint union). \\
Let $f\in S_{\mathbb F,d}$ be a degree $d$ homogeneous polynomial. \\
Let $H_{f}\subset \mathbb P^n$ be the closed subscheme of $\mathbb P^n$ defined by the homogeneous ideal $(f)\subset S_{\mathbb F}$.\\
Let $X$ be an $\mathbb F$-smooth projective equidimensional subscheme of $\mathbb P^n$. \\
Let $f_0,f_1,...,f_r\in S_{\mathbb F,hom}$ be homogeneous polynomials. \\
Let $X_i$ be the closed subscheme of $X$ equals $H_{f_i}\cap X$ (the scheme intersection). \\
For any $I\subset \{0,1,...,r\}$ write $X_I$ for the scheme intesection $\cap_{i\in I}X_i$ in $X$.
\end{notation}

Using \cite{Poo1} and \cite{Poo2} one can prove the following result
\begin{thm}(Bertini type theorem 3).
\label{Main3}
Let $X$ be a smooth projective equidimensional subscheme of $\mathbb P^n$.
Let $\underline x=\{x_1,...,x_l\}$ be a finite set of closed points in $X$.
Let $r\geq 1$ be the dimension of $X$.
There exist homogeneous polynomials
$f_0,f_1,...,f_r$ of degrees $e_0,e_1,...,e_r$ respectively
such that the subschemes $X_i$ enjoy the following properties: \\
0) $X_{0,...,r}=\emptyset$; \\
1) the schemes $X_{2,...,r}$ and $X_{1,...,r}$ are $\mathbb F$-smooth of dimensions $1$ and $0$ respectively; \\
3) $X_0$ and $X_1$ do not contain any point of $\underline x$; \\
3) for any $i>1$ the scheme $X_i$ contains the set $\underline x$; \\
4) for each $i=0,1,...,r-1$ the number $e_i$ divides $e_{i+1}$.

Let $Z$  be a closed subset in $X$ with $dim Z\leq r-1$. Then one can choose
$f_0,...,f_r\in S_{\mathrm{hom}}$ such that additionally
$X_{1,...,r}\cap Z=\emptyset$
and
$X_{0,2,...,r}\cap Z=\emptyset$
\end{thm}
\begin{notation}\label{notn:blowups_fin}
Let $X$, $r\geq 1$, $\underline x\subset X$, a closed subset $Z\subset X$,
homogeneous polynomials $f_0,f_1,...,f_r$ of degrees $e_0,e_1,...,e_r$ respectively
be as in Theorem \ref{Main3}. \\
For all $i,j=0,...,r$ with $j>i$ put $d_{i,j}=e_j/e_i$.\\
Write $\mathcal O(e)$ for $\mathcal O_{\mathbb P^n}(e)|_X$ (here $e\in \mathbb Z$). \\
Write $s_i\in \Gamma(X,\mathcal O(e_i))$ for $f_i|_X$.\\
Let $m\geq 0$, then write $\mathbb A^m$ for $\mathbb A^m_{\mathbb F}$. \\
Put $\dot q=(s_2/s^{d_{1,2}}_1,...,s_r/s^{d_{1,r}}_1): \dot X\to \mathbb A^{r-1}$;\\
For a Zariski open $S$ in $\mathbb A^{r-1}$ put $\dot X_S=\dot q^{-1}(S)$ and $\dot Z_S=Z\cap \dot X_S$.\\
Let $\mathbb P^{r,w}_{\mathbb F}$ be the weighted projective space
with homogeneous coordinates $[t_0:t_1:...:t_m]$ of weights $1,d_{0,1},...,d_{0,m}$ respectively.
Write $\mathbb P^{r,w}$ for $\mathbb P^{r,w}_{\mathbb F}$. \\
Let $\mathbb P^{r-1,w}_{\mathbb F}$ be the weighted projective space
with homogeneous coordinates $[x_1:...:x_m]$ of weights $1,d_{1,2},...,d_{1,m}$ respectively.
Write $\mathbb P^{r-1,w}$ for $\mathbb P^{r-1,w}_{\mathbb F}$. \\
Put $X_{s_i}=X-X_i$, $\mathbb P^{r,w}_{t_i}=\mathbb P^{r,w}-\{t_i=0\}$ and
$\mathbb P^{r-1,w}_{x_i}=\mathbb P^{r-1,w}-\{x_i=0\}$.\\
Write $\dot X$ for $X_{s_1}$ and $\dot {\mathbb P}^{r,w}$ for $\mathbb P^{r,w}_{t_1}$.\\
Note that $\mathbb P^{r-1,w}_{x_1}$ is the affine space $\mathbb A^{r-1}$.\\
The identification is given by sending
$[x_1:x_2:...:x_r]$ to $(x_2/x^{d_{1,2}}_1,...,x_r/x^{d_{1,r}}_1)$.\\
Note that $\mathbb P^{r,w}_{t_0}$ is the affine space $\mathbb A^{r}$.\\
The identification is given by sending
$[t_0:t_1:...:t_r]$ to $(t_1/t^{d_{0,1}}_0,...,t_r/t^{d_{0,r}}_0)$.
\end{notation}

\begin{prop}\label{weighted}
Under the Notation \ref{notn:blowups_fin}, 
the morphism
\begin{equation}\label{w}
\pi=[s_0:s_1:...:s_r]: X\to \mathbb P^{r,w}
\end{equation}
is well-defined and finite.
\end{prop}

\begin{proof}
One has $X=\cup^r_{i=0}X_{s_i}$, $\mathbb P^{r,w}=\cup^r_{i=0}\mathbb P^{r,w}_{t_i}$ and
$\pi^{-1}(\mathbb P^{r,w}_{t_i})=X_{s_i}$.
Since each $X_{s_i}$ is affine, the morphism $\pi$ is affine.
At the same time $\pi$ is projective. Thus, $\pi$ is finite.
\end{proof}

\begin{prop}\label{weighted2}
Under the hypotheses of Theorem \ref{Main3} and Notation \ref{notn:blowups} let $M\subset X$ be a closed subscheme
such that $X_{1,...,r}\cap M=\emptyset$. Then the morphism $[s_1:...:s_r]: M\to \mathbb P^{r-1,w}$
is finite.
Particularly, the morphism
$[s_1:...:s_r]: Z\to \mathbb P^{r-1,w}$
is finite.
\end{prop}

\begin{proof}
It is affine and projective.
\end{proof}

The following partial result will be proven at the end of in Section \ref{blowups}.
\begin{thm}\label{thm_Artin_2_fin}
Under the hypotheses of Theorem \ref{Main3}, Notation \ref{not:main} and \ref{notn:blowups_fin} consider the affine variety
$\dot X$ and the morphism
$\dot q: \dot X\to \mathbb A^{r-1}$.
For an open neighborhood $S\subset \mathbb A^{r-1}$ of the origin put $\dot X_S=\dot q^{-1}(S)$ and write $\dot q_S$ for
$\dot q|_{\dot X_S}: \dot X_S\to S$. There exists a Zariski open neighborhood $S\subset \mathbb A^{r-1}$ of the origin
such that  
\[
    \text{$\dot q_S: \dot X_S\to S$ is a weak elementary $\dot Z_S$-fibration.}
\]

\end{thm}

\section{Proof of the finite field case of Theorem \ref{Very_General}}\label{blowups}
The main aim of this section is to prove Theorem \ref{thm_Artin_2_fin}. We follow in this section Notation used in Section
\ref{sect:f_field_case}.
\begin{constr}
Define $\hat X$ as a closed subscheme of $X\times \mathbb P^{r-1,w}$
given by equations $s_jx^{d_{i,j}}_i=s^{d_{i,j}}_ix_j$ ($j>i$).
We regard $\hat X$ as a weighted blowup of $X$ at the subscheme $\underline w$.

If $X=\mathbb P^{r,w}$ is the weighted projective space as in Proposition \ref{weighted}
then $\hat {\mathbb P}^{r,w}$ is
a closed subscheme of $\mathbb P^{r,w}\times \mathbb P^{r-1,w}$
given by equations $t_jx^{d_{i,j}}_i=t^{d_{i,j}}_ix_j$.
We regard $\hat {\mathbb P}^{r,w}$ as a weighted blowup of $\mathbb P^{r,w}$ at the point $\mathbf 0:=[1:0:...:0]$.
\end{constr}
Projections $X\times \mathbb P^{r-1,w}\to \mathbb P^{r-1,w}$ and $\mathbb P^{r,w}\times \mathbb P^{r-1,w}\to \mathbb P^{r-1,w}$
induces morphisms $\hat q: \hat X\to \mathbb P^{r-1,w}$ and $\hat p: \hat {\mathbb P}^{r,w}\to \mathbb P^{r-1,w}$
allowing to consider the schemes $\hat X$ and $\hat {\mathbb P}^{r,w}$ as $\mathbb P^{r-1,w}$-schemes.
The morphism $\pi\times id: X\times \mathbb P^{r-1,w}\to \mathbb P^{r,w}\times \mathbb P^{r-1,w}$
induces a morphism $$\hat \pi: \hat X\to \hat {\mathbb P}^{r,w}.$$
This is a morphism of the $\mathbb P^{r-1,w}$-schemes. Put $\hat X_{x_1}=q^{-1}(\mathbb P^{r-1,w}_{x_1})$
and $\hat {\mathbb P}^{r,w}_{x_1}=p^{-1}(\mathbb P^{r-1,w}_{x_1})$.
Recall that $\mathbb P^{r-1,w}_{x_1}$ is the affine space $\mathbb A^{r-1}$.

Consider the weighted $\mathbb P^{1,w}$ with weighted homogeneous coordinates $[t_0:t_1]$ of weights $1,e_1/e_0$ respectively.
Put $d_1=e_1/e_0$. It is known that $\mathbb P^{1,w}$ is isomorphic to  $\mathbb P^{1}$ with homogeneous coordinates
$[t^{d_1}_0:t_1]$. Particularly, $\mathbb P^{1,w}$ is smooth. Consider a morphism
$$\phi: \hat {\mathbb P}^{r,w}_{x_1}\to \mathbb P^{1,w}\times \mathbb A^{r-1}$$
given by $([t_0:...:t_m],[x_1:...:x_m])\mapsto ([t_0:t_1],(x_2/x^{d_{1,2}}_1,..., x_m/x^{d_{1,m}}_1))$. Consider a morphism
$$\psi: \mathbb P^{1,w}\times \mathbb A^{r-1}\to \hat {\mathbb P}^{r,w}_{x_1}$$
given by $([t_0:t_1],(y_2,..., y_m))\mapsto ([t_0:t_1:t^{d_{1,2}}_1y_2:...:t^{d_{1,m}}_1y_m],[1:y_2:...:y_m])$.
\begin{lem}\label{phi_psi}
The morphisms $\phi$ and $\psi$ are well-defined and they are mutually inverse isomorphisms.
Moreover, they are isomorphisms of the $\mathbb A^{r-1}$-schemes.
\end{lem}

\begin{cor}\label{smooth}
The scheme $\hat {\mathbb P}^{r,w}_{x_1}$ is smooth.\\
The morphism $\hat p: \hat {\mathbb P}^{r,w}_{x_1}\to \hat {\mathbb P}^{r-1,w}_{x_1}=\mathbb A^{r-1}$ is smooth projective with
$\mathbb P^{1,w}\cong \mathbb P^{1}$ as a fiber.
\end{cor}

\begin{notation}\label{tau_and_sigma}
Projections $X\times \mathbb P^{r-1,w}\to X$ and $\mathbb P^{r,w}\times \mathbb P^{r-1,w}\to \mathbb P^{r,w}$
induces morphisms $\tau: \hat X\to X$ and $\sigma: \hat {\mathbb P}^{r,w}\to \mathbb P^{r,w}$ respectively.
Put $E_w=\tau^{-1}(\underline w)$, $E_0=\sigma^{-1}(\mathbf 0)$.
\end{notation}

\begin{rem}
Note that $\sigma \circ \hat \pi=\pi \circ \tau$.
Clearly, the morphisms
$\tau: \hat X-E_w\to X-\underline w$ and $\sigma: \hat {\mathbb P}^{r,w}-E_0\to \mathbb P^{r,w}-\{0\}$ are isomorphisms.

Note that $X=(X-\underline w)\cup X_{s_0}$ and $\mathbb P^{r,w}=(\mathbb P^{r,w}-\mathbf 0)\cup \mathbb P^{r,w}_{t_0}$.
\end{rem}

\begin{prop}\label{X_s_1}
The following are true \\
1) the morphisms
$\tau: \hat X_{x_1}-(E_w)_{x_1}\to \dot X$ and $\sigma: \hat {\mathbb P}^{r,w}_{x_1}-(E_0)_{x_1}\to \dot {\mathbb P}^{r,w}$ are isomorphisms; \\
2) the schemes $\hat X_{x_1}-(E_w)_{x_1}$, $\hat {\mathbb P}^{r,w}_{x_1}-(E_0)_{x_1}$, $\dot X$, and $\dot {\mathbb P}^{r,w}$ are smooth.
\end{prop}

\begin{proof}
The first assertion is clear. Prove the second one.
The scheme $X$ is smooth. Hence so are the schemes $\dot X$ and $\hat X_{x_1}-(E_w)_{x_1}$.
By Corollary \ref{smooth} the scheme $\hat {\mathbb P}^{r,w}_{x_1}$ is smooth. Hence so are the schemes
$\hat {\mathbb P}^{r,w}_{x_1}-(E_0)_{x_1}$ and $\dot {\mathbb P}^{r,w}$.
\end{proof}

\begin{lem}\label{flat_pi_0}
Let $X_{s_0}\xrightarrow{\pi_0} \mathbb P^{r,w}_{t_0}=\mathbb A^r$ be the base change of the morphism
$\pi: X\to \mathbb P^{r,w}$. Then $\pi_0$ is finite flat.
\end{lem}

\begin{proof}
By Proposition \ref{weighted} the morphism $\pi$ is finite surjective.
Thus the morphism $\pi_0$
is finite surjective as the base change of $\pi$.
Since $X_{s_0}$ and $\mathbb A^m$ are smooth it follows that
$\pi_0$ is flat. And it is also finite.
\end{proof}

\begin{lem}\label{flat}
The morphism $\hat \pi_{x_1}: \hat X_{x_1}\to \hat {\mathbb P}^{r,w}_{x_1}$ is finite flat.
\end{lem}

\begin{proof}
One has $\hat X_{x_1}=(\hat X-E_w)_{x_1}\cup (\hat X_{s_0})_{x_1}$,
$\hat {\mathbb P}^{r,w}_{x_1}=(\hat {\mathbb P}^{r,w}-E_0)_{x_1}\cup (\hat {\mathbb P}^{r,w}_{t_0})_{x_1}$.
Thus, it is sufficient to check that morphisms
$(\hat X-E_w)_{x_1}\xrightarrow{\hat \pi_{x_1}} (\hat {\mathbb P}^{r,w}-E_0)_{x_1}$
and
$(\hat X_{s_0})_{x_1}\xrightarrow{\hat \pi_{x_1}} (\hat {\mathbb P}^{r,w}_{t_0})_{x_1}$
are finite flat.

The morphism $\pi_0$ is finite flat by Lemma \ref{flat_pi_0}.
The morphism
$\hat X_{s_0}\xrightarrow{\hat \pi_0} \hat {\mathbb P}^{r,w}_{t_0}$
is a base change of $\pi_0$. Thus it is finite flat too.
Eventually the morphism
$\hat \pi_{x_1}: (\hat X_{s_0})_{x_1}\to (\hat {\mathbb P}^{r,w}_{t_0})_{x_1}$
is a base change of the morphism $\hat \pi_0$. Thus it is finite flat as well.

The morphism $X-w=\hat X-E_w\to \hat {\mathbb P}^{r,w}-E_0=\mathbb P^{r,w}-[1:0:...:0]$ is finite surjective.
Thus the morphism
$\hat X_{x_1}-(E_w)_{x_1}\to \hat {\mathbb P}^{r,w}_{x_1}-(E_0)_{x_1}$
is finite surjective. By the second item of Proposition \ref{X_s_1}
the source and the target are smooth schemes.
Thus, the morphism
$\hat X_{x_1}-(E_w)_{x_1}\to \hat {\mathbb P}^{r,w}_{x_1}-(E_0)_{x_1}$ is flat and finite.
Hence $\hat \pi_{x_1}$ is finite flat.
\end{proof}

\begin{lem}\label{hat_q_fl_proj}
The morphism $\hat q: \hat X_{x_1}\to \hat {\mathbb P}^{r-1,w}_{x_1}=\mathbb A^{r-1}$ is flat projective.
\end{lem}

\begin{proof}
We know that $\hat q=\hat p\circ \hat \pi_{x_1}$. Apply now Lemma \ref{flat} and Corollary \ref{smooth}.
\end{proof}

\begin{prop}\label{smooth2}
The scheme $\hat X_{x_1}$ is smooth.
\end{prop}

\begin{proof}
By Lemma \ref{flat_pi_0} the morphism $\pi_0$ is finite flat. The scheme
$\pi^{-1}_0(\{0\})=X_{1,...,m}=\underline w$ is smooth of dimension zero
by Theorem \ref{Main3}. Thus, $\pi_0$ is \'{e}tale over the origin $\{0\}$. So,
we can take a Zariski open $U\subset \mathbb A^m$ containing $\{0\}$
such that for $W=\pi^{-1}_0(U)$ the morphism $\pi_0: W\to U$ is finite \'{e}tale.
Then the morphisms
$\hat \pi_0: \hat W\to \hat U$ and $(\hat \pi_0)_{x_1}: \hat W_{x_1}\to \hat U_{x_1}$
are finite \'{e}tale as base change of $\pi_0$.
Note that $\hat U_{x_1}$ is open in $\hat {\mathbb P}^{r,w}_{x_1}$
and the latter is smooth by Corollary \ref{smooth}. Hence $\hat U_{x_1}$ is smooth
and so is $\hat W_{x_1}$ as an its \'{e}tale cover.

One has $\hat X_{x_1}=(\hat X-E_w)_{x_1}\cup \hat W_{x_1}$. By Proposition \ref{smooth}
$(\hat X-E_w)_{x_1}$ is smooth.
We know already that $\hat W_{x_1}$ is smooth.
Thus, $\hat X_{x_1}$ is smooth.
\end{proof}

Recall that
$\hat p: \hat {\mathbb P}^{r,w}_{x_1}\to \mathbb P^{r-1,w}_{x_1}=\mathbb A^{m-1}$
is smooth projective with the projective line $\mathbb P^1$ as fibres.
For the morphism
$\hat q: \hat X_{x_1}\to \mathbb P^{r-1,w}_{x_1}$
one has $\hat q=\hat \pi_{x_1}\circ \hat p$.

\begin{prop}\label{smooth_fibre}
Let $C={\hat q}^{-1}(\{0\})$. Then \\
1) $C$ is a smooth projective curve; \\
2) $\tau|_C: C\to X$ is a closed embedding; \\
3) $\tau(C)$ coincides with the smooth closed dimension one subscheme $X_{2,...,m}$ in $X$. \\
\end{prop}

\begin{proof}
Let $U$ and $W$ be as in the proof of Proposition \ref{smooth2}.
Prove the first assertion.
One has an open cover $C=(C-E_{\underline w})\cup (C\cap \hat W)$.
Note that $\tau: C-E_{\underline w}\to X_{2,...,m}-\underline w$ is an isomorphism.
Since $X_{2,...,m}$ is smooth, thus so is $C-E_{\underline w}$.

The $\hat \pi:\hat W\to \hat U$ is finite \'{e}tale. Hence so is the one
$C\cap \hat W\to (\mathbb P^1\times \{0\})\cap \hat U$. Since $\mathbb P^1$ is smooth,
hence so is $C\cap \hat W$. Thus $C$ is smooth. Since $C$ is closed in $X$ it is projective.

Prove the second assertion.
Put $l:=\hat p^{-1}(\{0\})$. By Corollary \ref{smooth}
$\sigma|_l: l\to \mathbb P^{r,w}$ identifies $l$ with the closed
subscheme $\mathbb P^{r,w}_{2,...,m}:=\{t_2=t_3=...=t_m=0\}$. Particularly, $\sigma|_l: l\to \mathbb P^{r,w}$ is a closed embedding and
$\sigma|_l: l\cap \hat U\to U$ is a closed embedding as well.
Put $\tau_C:=\tau|_C$. One has $X=(X-\underline w)\cup W$.
Thus $C=\tau^{-1}_C(X-\underline w)\cup \sigma^{-1}_C(W)=(C-E_{\underline w})\cup (C\cap \hat W)$.
Since $\tau_C: C-E_{\underline w}\to X-\underline w$ is a closed embedding it remains to check that
$\tau_C: (C\cap \hat W)\to W$ is a closed embedding. But this morphism is a base change
of the closed embedding $\sigma|_l: (l\cap \hat U)\to U$. Thus, the morphism $\tau_C: (C\cap \hat W)\to W$
is a closed embedding indeed.

Prove the third assertion.
The morphism $\tau: \hat X\to X$ is the base change of the one
$\sigma: \hat {\mathbb P}^{r,w}\to \mathbb P^{r,w}$ by means of $\pi$.
The closed embedding $j_C: C\hra \hat X$ is the base change of $j_l: l\hra \hat {\mathbb P}^{r,w}$ by means of $\hat \pi$.
Thus, $\tau_C: C\hra X$ is a base change of $\sigma|_l: l\hra \mathbb P^{r,w}$ by means of $\pi$.
Recall that $\sigma|_l$ identifies $l$ with $\mathbb P^{r,w}_{2,...,m}$. Thus
$\tau_C$ identifies $C$ with the subscheme $\pi^{-1}(\mathbb P^{r,w}_{2,...,m})=X_{2,...,m}$
of the scheme $X$.
%
%
\end{proof}

\begin{cor}\label{sm_proj}
There exists a Zariski open neighborhood $S$ of the origin $0\in \mathbb A^{r-1}$
such that for $\hat X_S:=\hat q^{-1}(S)$ the morphism $\hat q: \hat X_S\to S$ is smooth projective.
\end{cor}

\begin{notation}\label{j}
Write $j$ for the composition $\dot X\xrightarrow{\tau^{-1}} \hat X_{x_1}-(E_w)_{x_1}\hra \hat X_{x_1}$,
where $\tau$ is the isomorphism as in Proposition \ref{X_s_1}. Then $j$ is an open embedding.

Write $\hat Z$ for $\tau^{-1}(Z)\subset \hat X$. Since $Z\cap \underline w=\emptyset$ the morphism
$\tau$ identifies $\hat Z$ with $Z$. Also $j$ identifies $\dot Z:=Z_{s_1}$ with $\hat Z_{x_1}$.

Put $\dot q=\hat q\circ j: \dot X\to \mathbb P^{r-1,w}_{x_1}=\mathbb A^{r-1}$.

Write $i$ for the closed embedding $\mathbb A^{r-1}\times \underline w=(E_{\underline w})_{x_1}\hra \hat X_{x_1}$ and put $q_{\infty}=\hat q\circ i$.
\end{notation}

\begin{lem}\label{Z_finite_A}
The morphism $\dot q|_{\dot Z}: \dot Z\to \mathbb A^{m-1}$ is finite.\\
The morphism $q_{\infty}$ is finite \'{e}tale.
\end{lem}

\begin{proof}
The morphism $q|_{Z_{s_1}}$ is affine and projective. Thus it is finite.
\end{proof}

Consider the following commutative diagram
\begin{equation}
\label{SquareDiagram1}
    \xymatrix{
     \dot X\ar[drr]_{\dot q}\ar[rr]^{j}&&
\hat X_{x_1}\ar[d]_{\hat q}&&\mathbb A^{r-1}\times {\underline w}\ar[ll]_{i}\ar[lld]^{q_{\infty}} &\\
     && \mathbb A^{r-1}  &\\    }
\end{equation}
The morphism $\hat q$ is flat projective. The morphism $q_{\infty}$ is finite \'{e}tale. The morphism
$i$ is a closed embedding, the morphism $j$ is an open embedding identifying $\dot X$ with
$\hat X_{x_1}-i(\mathbb A^{r-1}\times {\underline w})$. The closed set $\underline x$ is in $X_{s_1}$.
If $Z\subset X$ is as in Theorem \ref{Main3}, then $\dot Z\subset X_{s_1}$ and $\dot Z$ is finite over $\mathbb A^{r-1}$
by Lemma \ref{Z_finite_A}.
\begin{thm}\label{w_elem_fib_1}
Let $S\subset \mathbb A^{r}$ be as in Corollary \ref{sm_proj}.
Then under Notation \ref{notn:blowups_fin}
the base change of the diagram \eqref{SquareDiagram1} by means of the open embedding $S\hra \mathbb A^{r-1}$
\begin{equation}
\label{SquareDiagram2}
    \xymatrix{
     \dot X_S\ar[drr]_{\dot q_S}\ar[rr]^{j_S}&&
\hat X_S\ar[d]_{\hat q_S}&&S\times {\underline w}\ar[ll]_{i}\ar[lld]^{\mathrm{pr}_S} &\\
     && S  &\\    }
\end{equation}
is a diagram of the weak elementary fibration $\dot q_S: \dot X_S\to S$ and
$\underline x\subset \dot X_S$. Moreover, if $Z\subset X$ is as in Theorem \ref{Main3}, then
$\dot q: \dot X_S\to S$ is a weak elementary $\dot Z_S$-fibration.
\end{thm}

\begin{proof}
This follows from Corollary \ref{sm_proj} and Lemma \ref{Z_finite_A}.
\end{proof}

Clearly, Theorem \ref{w_elem_fib_1} yields the following
\begin{cor}\label{cor_Artin_2}[= Theorem \ref{thm_Artin_2_fin}]
Under the notation and the hypotheses of Theorems \ref{Main3}, \ref{w_elem_fib_1} and Notation \ref{notn:blowups_fin}
consider the affine variety
$\dot X$ and the morphism
$\dot q: \dot X\to \mathbb A^{r-1}$.
Then for the open neighborhood $S\subset \mathbb A^{r-1}$ of the origin as in
Corollary \ref{sm_proj} and for the $\dot Z_S=Z\cap \dot X_S$
the morphism $\dot q_S: \dot X_S\to S$ is a weak elementary $\dot Z_S$-fibration.
\end{cor}

\section{An infinite field case of Theorem \ref{Very_General}}\label{sect:inf_field_case}
In this Section $k$ is an infinite field, $n\geq 1$. Let $\Pro^n_k$ be the projective space
with homogeneous coordinates $T_0,...,T_n$. Let $r\geq 1$ and $d\geq 1$. Consider
the vector space $S_{k,d}$ of degree $d$ homogeneous polynomials with coefficients in $k$ in variables $T_0,...,T_n$.
For each $i$ write $\bar T_i$ for $T_i|_X$. Then the homogeneous coordinate $k$-algebra $k[X]$ of $X$ is
$k[\bar T_0,...,\bar T_n]$.
\begin{notation}\label{not:main_inf}
Let $k$ be an infinite field, $n\geq 1$ be an integer.\\
Let $\mathbb P^n:=\mathbb P^n_k$ be the projective space over $k$. \\
Let $S_k = k[T_0,...,T_n]$ be the homogeneous coordinate ring of $\mathbb P^n$.\\
let $S_{k,d} \subset S_k$ be the $k$-subspace of homogeneous degree $d$ polynomials,\\
let $S_{k,\mathrm{hom}} = \sqcup^{\infty}_{d=0} S_{k,d}$ (the disjoint union). \\
Let $f\in S_{k,d}$ be a degree $d$ homogeneous polynomial. \\
Let $H_{f}\subset \mathbb P^n$ be the closed subscheme of $\mathbb P^n$ defined by the homogeneous ideal $(f)\subset S_k$.\\
Let $X$ be an $k$-smooth projective equidimensional subscheme of $\mathbb P^n$ of dimension $r$. \\
Let $W\subset S_{k,d}\subset S_{k,\mathrm{hom}}$ be a dimension $r$ vector $k$-subspace.\\
Let ${\bf f}=\{f_1,...,f_r\}$ be a free $k$-basis of $W$.\\
Let $X_i$ be the closed subscheme of $X$ equals $H_{f_i}\cap X$ (the scheme intersection). \\
For any $I\subset \{0,1,...,r\}$ write $X_I$ for the scheme intesection $\cap_{i\in I}X_i$ in $X$.\\
Write $\underline w=X_{1,...,r}$ for the closed subscheme in $X$.\\
Write $\mathcal O(d)$ for $\mathcal O_{\mathbb P^n}(d)|_X$.\\
For the $f_i$ write $s_i\in \Gamma(X,\mathcal O(d))$ for $f_i|_X$.\\
The closed subscheme $\underline w=X_{1,...,r}\subset X$ depends only on $W$.
\end{notation}

Extending the aurguments from~\cite[Exp. XI, Thm. 2.1]{LNM305} we get the following Bertini type result
(see \cite{PSV})
\begin{thm}
\label{GeneralSection_bar k} Let $k$ be an algebraically closed and let
$X\subset \Pro^n_k$ be a smooth closed subscheme of pure
dimension $r\geq 1$.
Let ${\bf x}\in X$ be a finite subset of
closed points.
Under the Notation \ref{not:main_inf}
for each integer $d\gg 0$
for {\bf a generic} dimension $r$ vector subspace $W$ in $S_{k,d}$
and {\bf a generic} free $k$-basis
${\bf f}=\{f_1,\dots,f_{r}\}$ of $W$
the following is true \\
(i) the scheme $\underbar w$ is $k$-smooth of dimension zero and $\underbar w\cap X_0=\emptyset$; \\
(ii) $\pi|_{\bf x}: {\bf x}\to \P^r$ is a closed embedding, where $\pi=[s_0:...:s_r]: X\to \P^r$; \\
(iii) for each $x\in {\bf x}$ one has $s_2(x)\neq 0\neq s_1(x)$; particularly, ${\bf x}\cap \underbar w=\emptyset$ and ${\bf x}\subset X_{s_2}$; \\
(iv) for each $x\in {\bf x}$ the closed subscheme of $X$ defined by equations
$\{s_1(x)s_i=s_i(x)s_1\}$, where $i\in \{2,...,r\}$,
is smooth of dimension one.

Let $Z\subset X$ be a closed subset with $dim Z < r$.
Then one can choose the $W$ and the ${\bf f}$ above
such that the property (i) to (iv) does hold and the following is true \\
(v) $\underbar w\cap Z=\emptyset$; particularly, $Z\subset X-\underbar w$
and the map $q=[s_1:...:s_r]: Z\to \P^{r-1}$ is finite.
\end{thm}
As an immediate consequence of this Theorem we get the following
\begin{cor}\label{GeneralSection_inf k}
Let $k$ be an infinite field. Let
$X\subset \Pro^n_k$ be a $k$-smooth closed subscheme of pure
dimension $r$.
Let ${\bf x}\subset X$ be a finite suet of closed points.
Under the Notation \ref{not:main_inf}
for each integer $d\gg 0$
for {\bf a generic} dimension $r$ vector subspace $W$ in $S_{k,d}$
and {\bf a generic} its free $k$-basis
${\bf f}=\{f_1,\dots,f_{r}\}$
the following is true \\
(i) the scheme $\underbar w:=X_{1,...,r}$ is $k$-smooth of dimension zero and $\underbar w\cap X_0=\emptyset$;\\
(ii) $\pi|_{\bf x}: {\bf x}\to \P^r$ is a closed embedding, where $\pi=[s_0:...:s_r]: X\to \P^r$; \\
(iii) ${\bf x}\subset X_{s_1}\cap X_{s_2}$; particularly, ${\bf x}\cap \underbar w=\emptyset$ and ${\bf x}\subset X_{s_2}$; \\
(iv) for each point $x\in {\bf x}\otimes_k \bar k$ the closed subscheme of $X\otimes_k \bar k$ defined by equations
$\{s_1(x)s_i=s_i(x)s_1\}$, where $i\in \{2,...,r\}$,
is smooth of dimension one.

Let $Z\subset X$ be a closed subset with $dim Z < r$.
Then one can choose the $W$ and the ${\bf f}$ above
such that the property (i) to (iv) does hold and the following is true \\
(iv) $\underbar w\cap Z=\emptyset$; particularly, $Z\subset X-\underbar w$
and the map $q|_Z: Z\to \P^{r-1}_k$ is finite.
\end{cor}
{\bf For the rest of the section the field $k$ is infinite.
}
We use below in this section Notation \ref{not:main_inf}.
\begin{notation}\label{notn:blowups}
Let $X$, $r\geq 1$, ${\bf x}\subset X$, a closed subset $Z\subset X$, \\
the integer $d\gg 0$, \\
{\bf the generic} dimension $r$ vector subspace $W$ in $S_{k,d}$
and \\
{\bf the generic} free $k$-basis
${\bf f}=\{f_1,\dots,f_{r}\}$ of $W$
be as in Theorem \ref{GeneralSection_bar k}. \\
Write $\mathcal O(d)$ for $\mathcal O_{\mathbb P^n}(d)|_X$. \\
Write $s_i\in \Gamma(X,\mathcal O(e_i))$ for $f_i|_X$.\\
Let $m\geq 0$, then write $\mathbb A^m$ for $\mathbb A^m_k$. \\
Put $\dot q=(s_2/s_1,...,s_r/s_1): \dot X\to \mathbb A^{r-1}$;\\
For a Zariski open $S$ in $\mathbb A^{r-1}$ put $\dot X_S=\dot q^{-1}(S)$ and $\dot Z_S=Z\cap \dot X_S$.\\
Let $\mathbb P^{r}:=\mathbb P^{r}_{k}$ be the projective space
with homogeneous coordinates $[t_0:t_1:...:t_r]$.\\
Let $\mathbb P^{r-1}:=\mathbb P^{r-1}_k$ be the projective space
with homogeneous coordinates $[x_1:...:x_r]$.\\
Put $X_{s_i}=X-X_i$, $\mathbb P^{r}_{t_i}=\mathbb P^{r}-\{t_i=0\}$ and
$\mathbb P^{r-1}_{x_i}=\mathbb P^{r-1}-\{x_i=0\}$.\\
Write $\dot X$ for $X_{s_1}$ and $\dot {\mathbb P}^{r}$ for $\mathbb P^{r}_{t_1}$.\\
Note that $\mathbb P^{r-1}_{x_1}$ is the affine space $\mathbb A^{r-1}$.\\
The identification is given by sending
$[x_1:x_2:...:x_r]$ to $(x_2/x_1,...,x_r/x_1)$.\\
Note that $\mathbb P^{r}_{t_0}$ is the affine space $\mathbb A^{r}$.\\
The identification is given by sending
$[t_0:t_1:...:t_r]$ to $(t_1/t_0,...,t_r/t_0)$.
\end{notation}

\begin{prop}\label{non_weighted}
Under the Notation \ref{notn:blowups}
the morphism
\begin{equation}\label{w}
\pi=[s_0:s_1:...:s_r]: X\to \mathbb P^{r}
\end{equation}
is well-defined and finite.
\end{prop}

\begin{proof}
One has $X=\cup^r_{i=0}X_{s_i}$, $\mathbb P^{r}=\cup^r_{i=0}\mathbb P^{r}_{t_i}$ and
$\pi^{-1}(\mathbb P^{r}_{t_i})=X_{s_i}$.
Since each $X_{s_i}$ is affine, the morphism $\pi$ is affine.
At the same time $\pi$ is projective. Thus, $\pi$ is finite.
\end{proof}

\begin{prop}\label{non_weighted2}
Under the hypotheses of Theorem \ref{Main3} and Notation \ref{notn:blowups} let $M\subset X$ be a closed subscheme
such that $\underline w\cap M=\emptyset$. Then the morphism $[s_1:...:s_r]: M\to \mathbb P^{r-1}$
is finite.
Particularly, the morphism
$[s_1:...:s_r]: Z\to \mathbb P^{r-1}$
is finite.
\end{prop}

\begin{proof}
It is affine and projective.
\end{proof}

The following partial result will be proven at the end of in Section \ref{usual_blowups}.
\begin{thm}\label{thm_Artin_2}
Under the hypotheses of Theorem \ref{GeneralSection_bar k}, Notation \ref{not:main_inf} and \ref{notn:blowups}
for each integer $d\gg 0$,
for {\bf a generic} dimension $r$ vector subspace $W$ in $S_{k,d}$
and {\bf a generic} free $k$-basis
${\bf f}=\{f_1,\dots,f_{r}\}$ of $W$
the following is true. \\
For the affine variety $\dot X$ and the morphism
$\dot q: \dot X\to \mathbb A^{r-1}$
there exists a Zariski open neighborhood $S\subset \mathbb A^{r-1}$ of the
finite set $\dot q({\bf x})$
such that 
\[
  \text{$\dot q_S: \dot X_S\to S$ is a weak elementary $\dot Z_S$-fibration. } 
\]
\end{thm}

\section{Proof of the infinite field case of Theorem \ref{Very_General}}\label{usual_blowups}
The main aim of this section is to prove Theorem \ref{thm_Artin_2}. We follow in this section Notation used in Section
\ref{sect:inf_field_case}.
\begin{constr}\label{hat X}
Under Notation \ref{notn:blowups}
Define $\hat X$ as a closed subscheme of $X\times \mathbb P^{r-1}$
given by equations $s_jx_i=s_ix_j$.
We regard $\hat X$ as the blowup of $X$ at the \'{e}tale over $\mathrm{Spec} (k)$ subscheme $\underline w$ of $X$.

If $X=\mathbb P^{r}$ is the projective space as in Proposition \ref{non_weighted}
then $\hat {\mathbb P}^{r}$ is
a closed subscheme of $\mathbb P^{r}\times \mathbb P^{r-1}$
given by equations $t_jx_i=t_ix_j$.
We regard $\hat {\mathbb P}^{r}$ as a weighted blowup of $\mathbb P^{r}$ at the point $\mathbf 0:=[1:0:...:0]$.
\end{constr}
The projection $X\times \mathbb P^{r-1}\to \mathbb P^{r-1}$
induces a morphism
$\hat q: \hat X\to \mathbb P^{r-1}$
allowing to consider the scheme $\hat X$
as a $\mathbb P^{r-1}$-scheme.
Let $[x_1:...:x_r]$ be homogeneous coordinates on $\mathbb P^{r-1}_k$.
Put
$\mathbb P^{r-1}_{x_i}=\mathbb P^{r-1}-\{x_i=0\}$.
Clearly, $\mathbb P^{r-1}_{x_1}\cong \mathbb A^{r-1}_k$.
Just identify
$[x_1:x_2:...:x_r]$ with $(x_2/x_1,...,x_r/x_1)$.
Put $\hat X_{x_1}=\hat q^{-1}(\mathbb P^{r-1}_{x_1})$.
The $X\subset \P^n_k$, the $W$ and the ${\bf f}=\{f_1,...,f_r\}$ are as in
Corollary \ref{GeneralSection_inf k}. Particularly,
the $\underbar w$ is finite \'{e}tale over $\mathrm{Spec} (k)$.
Let $[x_1:...:x_r]$ as above be the homogeneous coordinates on $\mathbb P^{r-1}_k$,
$\mathbb P^{r-1}_{x_i}=\mathbb P^{r-1}-\{x_i=0\}$ and
$\mathbb P^{r-1}_{x_1}\cong \mathbb A^{r-1}_k$.
Just identify
$[x_1:x_2:...:x_r]$ with $(x_2/x_1,...,x_r/x_1)$.
As above write $\hat X_{x_1}$ for $\hat q^{-1}(\mathbb P^{r-1}_{x_1})$.


\begin{notation}\label{tau_and_sigma}
The projection $X\times \mathbb P^{r-1}\to X$
and $\mathbb P^{r}\times \mathbb P^{r-1}\to \mathbb P^{r}$
induces a morphism
$\tau: \hat X\to X$
and
$\sigma: \hat {\mathbb P}^{r}\to \mathbb P^{r}$ respectively.
Write
$\mathbf 0$ for $[1:0:...:0]$ in $\mathbb P^{r}$.
Put $E_w=\tau^{-1}(\underbar w)$
and
$E_0=\sigma^{-1}(\mathbf 0)$.\\
Note that $E_w=\P^{r-1}\times_{pt} \underbar w$ where $pt$ stands for $\mathrm{Spec}(k)$; \\
Write $\hat X\xleftarrow{i} E_w=\P^{r-1}\times_{pt} \underbar w$ \ for the closed embedding.\\
Write $X_{s_i}$ for $X-X_i$ and write $\dot X$ for $X_{s_1}$.\\
Let $Z\subset X$ be a closed subset with $Z\subset X-\underbar w$.
Put $\dot Z=Z\cap \dot X$.
\end{notation}

\begin{rem}
Clearly, the morphism
$\tau: \hat X-E_w\to X-\underline w$
is an isomorphisms. \\
Put $\hat Z$ for $\tau^{-1}(Z)$ and $Z_{x_1}=\hat Z\cap \hat X_{x_1}$.
\end{rem}

\begin{prop}\label{X_s_1}
Put $\tau_1=\tau|_{\hat X_{x_1}-(E_w)_{x_1}}: \hat X_{x_1}-(E_w)_{x_1}\to \dot X$. Then \\
1) the morphism
$\tau_1: \hat X_{x_1}-(E_w)_{x_1}\to \dot X$
is an isomorphism; \\
2) the schemes $\hat X_{x_1}-(E_w)_{x_1}$ and
$\dot X$
are $k$-smooth.\\
3) the morphism
$\tau_1|_{\hat Z_{x_1}}: \hat Z_{x_1}\to \dot Z$
is an isomorphism.
\end{prop}

\begin{notation}\label{j=tau -1}
Write $j: \dot X\hra \hat X_{x_1}$ for $\dot X\xrightarrow{\tau^{-1}} \hat X_{x_1}-(E_w)_{x_1}\hra \hat X_{x_1}$.\\
Clearly, $j|_{\dot Z}$ identifies $\dot Z$ with $\hat Z_{x_1}$; \\
Put $\dot q=\hat q\circ j: \dot X\to \mathbb A^{r-1}_k$; clearly, $\dot q=(s_2/s_1,...,s_r/s_1)$; \\
For an open $S$ in $\mathbb A^{r-1}_k$ write \\
$\dot X_S$ for $\dot q^{-1}(S)$, $\hat X_S$ for $\hat q^{-1}(S)$; $\dot Z_S$ for $Z\cap \dot X_S$; $\hat Z_S$ for $\hat Z\cap \hat X_S$; \\
clearly, $\dot X_S=j^{-1}(\hat X_S)$; \\
put $\dot q_{S}=\dot q|_{\dot X_S}: \dot X_S\to S$ and $\hat q_S=\hat q|_{\hat X_S}: \hat X_S\to S$ and $j_S=j|_{\dot X_S}: \dot X_S\hra \hat X_S$; \\
clearly, $j_S|_{\dot Z_S}$ identifies $\dot Z_S$ with $\hat Z_S$; \\
clearly, $\dot q_S=\hat q_S\circ j_S: \dot X_S\to S$; \\
put $\hat X_S\leftarrow S\times_{pt} \underbar w: \ i_S=i\times_{pt} \underbar w$; \\
clearly, $\hat q_S\circ i_S=pr_S: S\times_{pt} \underbar w\to S$.
\end{notation}

\begin{thm}\label{w_el_fib_inf_field}
Let $k$,
$X\subset \Pro^n_k$, $r$, $x\in X$, $d\gg 0$,
{\bf a generic} dimension $r$ vector subspace $W$ in $S_{k,d}$
{\bf a generic} its free $k$-basis
${\bf f}=\{f_1,\dots,f_{r}\}$
be as in of Corollary \ref{GeneralSection_inf k}.
Let $S$ be the open neighborhood in $\mathbb A^{r-1}_k$ of the point
$\dot q(x)\in \P^{r-1}_{x_1}=\mathbb A^{r-1}_k$
as in the item (iii) of Corollary \ref{GeneralSection_inf k}.
Then under Notation \ref{tau_and_sigma} and  Notation \ref{j=tau -1} the commutative diagram
\begin{equation}
\label{SquareDiagram_0}
    \xymatrix{
     \dot X_S\ar[drr]_{\dot q_S}\ar[rr]^{j_S}&&
\hat X_S\ar[d]^{\hat q_S}&&S\times_{pt} \bar w\ar[ll]_{i_S}\ar[lld]^{pr_S} &\\
     && S  &\\    }
\end{equation}
is a diagram of a weak elementary fibration.

Let $Z\subset X$ be a closed subset with $\dim Z < r$.
Then one can choose the $W$ and the ${\bf f}$ above such that
the properties (i) to (iii) as in Corollary \ref{GeneralSection_inf k}
does hold and
$\underbar w\cap Z=\emptyset$.
In this case the diagram \eqref{SquareDiagram_0} is a diagram of the weak elementary $\dot Z_S$-fibration
$\dot q_S: \dot X_S\to S$.
\end{thm}
\begin{proof}[Proof of Theorem \ref{w_el_fib_inf_field}]
Take $(e_0,e_1,...,e_r)$ from Section \ref{blowups}
equal $(1,1,...,1)$.
And repeat literally arguments proving
Theorem \ref{w_elem_fib_1}.
\end{proof}

Clearly, Theorem \ref{w_el_fib_inf_field} yields the following
\begin{cor}\label{cor: w_el_fib_inf_field}[=Theorem \ref{thm_Artin_2}]
Under the hypotheses of Theorem \ref{w_el_fib_inf_field}
let $S$ be the open neighborhood in $\mathbb A^{r-1}_k$ of the point
$\dot q(x)\in \mathbb A^{r-1}_k$
as in the item (iii) of Corollary \ref{GeneralSection_inf k}.
Then $\dot q_S=(s_2/s_1,...,s_r/s_1): \dot X_S\to S$ and the morphism
$\dot q_S$
is a weak elementary fibration.

Let $Z\subset X$ be a closed subset with $\dim Z < r$ and suppose
the $W$ and the ${\bf f}$ above subjects to conditions
(i) to (iii) as in Corollary \ref{GeneralSection_inf k}
and additionally
$\underbar w\cap Z=\emptyset$.
In this case the morphism $\dot q_S$
is a weak elementary $\dot Z_S$-fibration.
\end{cor}

\section{The case of a finite field $k(b)$ }\label{sect_DVR-first}
The main aim of this section is to state Theorem \ref{cor: w_el_fib_fin_A}.\\
$A$ is a regular local ring; $m\subseteq A$ is its maximal ideal;\\
$K$ is the fraction field of $A$;
$B=\mathrm{Spec}(A)$ and $b\in B$ is its closed point;\\
$k(b)$ is the residue field $A/m$ and suppose {\bf it is finite} in this section;\\
$p>0$ is the characteristic of the field $k(b)$; \\
it is supposed in this notes that the field $K$ has characteritic zero;\\
$m\geq 1$ is an integer;\\
For an $A$-scheme $Y$ we write $Y_b$ for the closed fibre of the structure morphism $Y\to B$.\\
For an $A$-scheme morphism $q: Y'\to Y$ write $q_b: Y'_b\to Y_b$ for the morphism $q\times_B b$;\\
All schemes in this section are $B$-schemes and we write $Y\times Y'$ for $Y\times_B Y'$. \\
All morphisms in this section are $B$-morphisms.
\begin{notation}\label{not:main_A}
Let $\mathbb P^n:=\mathbb P^n_{B}$ be the projective space over $B$. \\
Let $S_{A} = A[T_0,...,T_n]$ be the homogeneous coordinate ring of $\mathbb P^n_B$.\\
let $S_{A,d} \subset S_{A}$ be the $A$-submodule of homogeneous degree $d$ polynomials,\\
let $S_{A,\mathrm{hom}} = \sqcup^{\infty}_{d=0} S_{A,d}$ (the disjoint union). \\
Let $F\in S_{A,d}$ be a degree $d$ homogeneous polynomial. \\
Let $H_{F}\subset \mathbb P^n_B$ be the closed subscheme of $\mathbb P^n_B$ defined by the homogeneous ideal $(F)\subset S_{A}$.\\
Let $X$ be a $B$-smooth closed irreducible subscheme of $\mathbb P^n_B$ of relative dimension $r\geq 1$. \\
Let $X\subset \mathbb P^n_B$ be an irreducible $B$-smooth closed subscheme of relative dimension $r$.\\
Let ${\bf x}=\{x_1,...,x_l\}$ be a finite set of closed points in $X$.\\
We will write $\mathcal O(e)$ for $\mathcal O_{\mathbb P^n}(e)|_X$; \\
Let $Z\subset X$ be a closed subset such that $\dim Z_b\leq r-1$.\\
For each $F\in S_{A,\mathrm{hom}}$ put $X_F=X-H_{F}$ in $X$; \\
Let $\mathbb P^n_{b}$ be the projective space over $b$. \\
Let $S_{k(b)} = k(b)[T_0,...,T_n]$ be the homogeneous coordinate ring of $\mathbb P^n_b$.\\
Let $S_{k(b),d} \subset S_{k(b)}$ be the $k(b)$-subspace of homogeneous degree $d$ polynomials,\\
let $S_{k(b),\mathrm{hom}} = \sqcup^{\infty}_{d=0} S_{k(b),d}$ (the disjoint union). \\\\

Let $f_0,...f_r\in S_{k(b),\mathrm{hom}}$ be as in Theorem \ref{Main3} with respect to \\
the subscheme $X_b\subset \mathbb P^n_b$, its finite subset ${\bf x}$ and its closed subset $Z_b\subset X_b$. \\
Let $e_0,...,e_r$ be the degrees of $f_0,...f_r$, then for each $i=0,...,r-1$ the $e_i$ divides $e_{i+1}$.\\
For all $i,j=1,...,r$ with $j>i$ put $d_{i,j}=e_j/e_i$.\\
Let $F_0,...,F_r\in S_{A,\mathrm{hom}}$ be such that $F_i \mod \ m\equiv f_i$.\\
Put $X_i=X\cap H_{F}$ (the scheme intersection); it is a closed subscheme of $X$; \\
For any $I\subset \{0,1,...,r\}$ write $X_I$ for the scheme intesection $\cap_{i\in I}X_i$ in $X$.\\
Put $\mathbb W=X_{1,...,r}$.\\
In this case $\mathbb W_b\cap Z_b=\emptyset$
and $\mathbb W\cap Z=\emptyset$; \\
Write $s_i\in \Gamma(X,\mathcal O(e_i))$ for $F_i|_X$.\\
Put $X_{s_i}=X_{F_i}$ and $\dot X=X_{s_1}$.\\
Put $\dot q=(s_2/s^{d_{1,2}}_1,...,s_r/s^{d_{1,r}}_1): \dot X\to \mathbb A^{r-1}_B$;\\
For a Zariski open $S$ in $\mathbb A^{r-1}_B$ put $\dot X_S=\dot q^{-1}(S)$ and $\dot Z_S=Z\cap \dot X_S$.\\
Write $\dot q_S$ for $\dot q|_{\dot X_S}: \dot X_S\to S$.\\
Let $\mathbb P^{r,w}_{B}$ be the weighted projective space
with homogeneous coordinates $[t_0:t_1:...:t_m]$ of weights $1,d_{0,1},...,d_{0,m}$ respectively.
Write $\mathbb P^{r,w}$ for $\mathbb P^{r,w}_{B}$. \\
Let $\mathbb P^{r-1,w}_{B}$ be the weighted projective space
with homogeneous coordinates $[x_1:...:x_m]$ of weights $1,d_{1,2},...,d_{1,m}$ respectively.
Write $\mathbb P^{r-1,w}$ for $\mathbb P^{r-1,w}_{B}$. \\
Put $X_{s_i}=X_{F_i}$ and also consider $\mathbb P^{r,w}_{t_i}$ and $\mathbb P^{r-1,w}_{x_i}$.\\
Write $\dot X$ for $X_{s_1}$ and $\dot {\mathbb P}^{r,w}$ for $\mathbb P^{r,w}_{t_1}$.\\
Note that $\mathbb P^{r-1,w}_{x_1}$ is the affine space $\mathbb A^{r-1}_B$.\\
The identification is given by sending
$[x_1:x_2:...:x_r]$ to $(x_2/x^{d_{1,2}}_1,...,x_r/x^{d_{1,r}}_1)$.\\
Note that $\mathbb P^{r,w}_{t_0}$ is the affine space $\mathbb A^{r}_B$.\\
The identification is given by sending
$[t_0:t_1:...:t_r]$ to $(t_1/t^{d_{0,1}}_0,...,t_r/t^{d_{0,r}}_0)$.\\
$X_{0,1,...,r}=\emptyset$ by the choice of $f_0,...f_r\in S_{k(b),\mathrm{hom}}$.
\end{notation}

\begin{prop}\label{weighted_A}
Under the hypotheses of Theorem \ref{Main3} and Notation \ref{not:main_A} we know that $X_{0,1,...,r}=\emptyset$. Thus,
the morphism
\begin{equation}\label{w}
\pi=[s_0:s_1:...:s_r]: X\to \mathbb P^{r,w}_B
\end{equation}
is well-defined. We state that it is finite and surjective.
\end{prop}

\begin{proof}
One has $X=\cup^r_{i=0}X_{s_i}$, $\mathbb P^{r,w}=\cup^r_{i=0}\mathbb P^{r,w}_{t_i}$ and
$\pi^{-1}(\mathbb P^{r,w}_{t_i})=X_{s_i}$.
Since each $X_{s_i}$ is affine, the morphism $\pi$ is affine.
At the same time $\pi$ is projective. Thus, $\pi$ is finite.
\end{proof}

\begin{prop}\label{weighted2_A}
Under the hypotheses of Theorem \ref{Main3} and Notation \ref{not:main_A} let $M\subset X$ be a closed subscheme
such that $X_{1,...,r}\cap M=\emptyset$. Then the morphism $[s_1:...:s_r]: M\to \mathbb P^{r-1,w}_B$
is finite.
Particularly, the morphism
$[s_1:...:s_r]: Z\to \mathbb P^{r-1,w}_B$
is finite.
\end{prop}

\begin{proof}
It is affine and projective.
\end{proof}

The following result will be proven at the end of in Section \ref{blowups_A_finn}.
\begin{thm}\label{cor: w_el_fib_fin_A}
Under the hypotheses of Theorem \ref{Main3} and Notation \ref{not:main_A}
consider the affine $B$-scheme
$\dot X$ and the morphism
$\dot q: \dot X\to \mathbb A^{r-1}_B$.
Then there exists a Zariski open neighborhood $S\subset \mathbb A^{r-1}_B$ of the point $(0,b)$
such that
the morphism 
$\dot q_S: \dot X_S\to S$
is a weak elementary $\dot Z_S$-fibration.
%
\end{thm}

\section{Proof of the finite $k(b)$-case of Theorem \ref{Very_General}}\label{blowups_A_finn}
The main aim of this section is to prove Theorem \ref{cor: w_el_fib_fin_A}.
We follow in this section Notation used in Section  \ref{sect_DVR-first}. Particularly, \\
Let $X\subset \mathbb P^n_B$ be an irreducible $B$-smooth closed subscheme of relative dimension $r\geq 1$.
Let ${\bf x}=\{x_1,...,x_l\}$ be a finite set of closed points in $X$.\\
We will write $\mathcal O(e)$ for $\mathcal O_{\mathbb P^n}(e)|_X$; \\
Let $Z\subset X$ be a closed subset such that $dim Z_b\leq r-1$.\\
For each $F\in S_{A,\mathrm{hom}}$ put $X_F=X-H_{F}$ in $X$; \\\\
Let $f_0,...f_r\in S_{k(b), \mathrm{hom}}$ be as in Theorem \ref{Main3} with respect to \\
the subscheme $X_b\subset \mathbb P^n_b$, its finite subset ${\bf x}$ and its closed subset $Z_b\subset X_b$. \\
Let $F_0,...,F_r\in S_{A,\mathrm{hom}}$ be such that $F_i \ mod \ m\equiv f_i$.\\
Put $X_i=X\cap H_{F_i}$ (the scheme intersection); it is a closed subscheme of $X$; \\
For any $I\subset \{0,1,...,r\}$ write $X_I$ for the scheme intesection $\cap_{i\in I}X_i$ in $X$.\\
Recall that $\mathbb W:=X_{1,...,r}$; \\ 
Recall also that $\mathbb W_b\cap Z_b=\emptyset$
and $\mathbb W\cap Z=\emptyset$ by the choice of $f_0,...f_r\in S_{k(b), \mathrm{hom}}$.\\
Recall also that $X_{0,1,...,r}=\emptyset$ by the choice of $f_0,...f_r\in S_{k(b), \mathrm{hom}}$.
\begin{notation}\label{notn:blowups_A_fin}
By the above assumptions on $f_0,...f_r\in S_{k(b), \mathrm{hom}}$ the scheme $\mathbb W$ is finite \'{e}tale over $B$. 
Since $X_{0,1,...,r}=\emptyset$ there is the well defined morphism
$$\pi=[s_0:...:s_r]: X\to \mathbb P^{r,w}_B.$$
as in Proposition \ref{weighted2_A}. Moreover, $\pi$ is
finite surjective by Proposition \ref{weighted2_A}.
\end{notation}

\begin{constr}
Define $\hat X$ as a closed subscheme of $X\times \mathbb P^{r-1,w}$
given by equations $s_jx^{d_{i,j}}_i=s^{d_{i,j}}_ix_j$ ($j>i$).
We regard $\hat X$ as a weighted blowup of $X$ at the subscheme $\mathbb W$.

If $X=\mathbb P^{r,w}$ is the weighted projective space as in Proposition \ref{weighted}
then $\hat {\mathbb P}^{r,w}$ is
a closed subscheme of $\mathbb P^{r,w}\times \mathbb P^{r-1,w}$
given by equations $t_jx^{d_{i,j}}_i=t^{d_{i,j}}_ix_j$.
We regard $\hat {\mathbb P}^{r,w}$ as a weighted blowup of $\mathbb P^{r,w}$ at the $B$-point $\mathbf 0:=[1:0:...:0]$.
\end{constr}

To state Theorem
\ref{el_fib_DVR_fin_k_v} we need more preparations.
Projections
$X\times \mathbb P^{r-1,w}\to \mathbb P^{r-1,w}$
and
$\mathbb P^{r,w}\times \mathbb P^{r-1,w}\to \mathbb P^{r-1,w}$
induces morphisms
$\hat q: \hat X\to \mathbb P^{r-1,w}$
and
$\hat p: \hat {\mathbb P}^{r,w}\to \mathbb P^{r-1,w}$
allowing to consider the schemes
$\hat X$ and $\hat {\mathbb P}^{r,w}$
as
$\mathbb P^{r-1,w}$-schemes.
The morphism
$\pi\times id: X\times \mathbb P^{r-1,w}\to \mathbb P^{r,w}\times \mathbb P^{r-1,w}$
induces a morphism
$$\hat \pi: \hat X\to \hat {\mathbb P}^{r,w}.$$
This is a morphism of the
$\mathbb P^{r-1,w}$-schemes.
Put
$\hat X_{x_1}=\hat q^{-1}(\mathbb P^{r-1,w}_{x_1})$
and
$\hat {\mathbb P}^{r,w}_{x_1}=\hat p^{-1}(\mathbb P^{r-1,w}_{x_1})$.
Note that
$\mathbb P^{r-1,w}_{x_1}$ is the affine space $\mathbb A^{r-1}_B$.
The identification is given by \\
$[x_1:x_2:...:x_r]\mapsto (x_2/x^{d_{1,2}}_1,...,x_r/x^{d_{1,r}}_1)$.
Thus, we have projective $B$-morphisms
$\hat q: \hat X_{x_1}\to \mathbb A^{r-1}_B$, $\hat p: \hat {\mathbb P}^{r,w}_{x_1}\to \mathbb A^{r-1}_B$
and
$\hat \pi: \hat X_{x_1}\to \hat {\mathbb P}^{r,w}_{x_1}$
with
$\hat q=\hat p\circ \hat \pi_{x_1}$.
\begin{notation}\label{tau_and_sigma_DVR}
The projections $X\times \mathbb P^{r-1,w}\to X$
and $\mathbb P^{r,w}\times \mathbb P^{r-1,w}\to \mathbb P^{r,w}$
induces morphisms
$\tau: \hat X\to X$
and
$\sigma: \hat {\mathbb P}^{r,w}\to \mathbb P^{r,w}$ respectively. \\
Write
$\mathbf 0$ for the closed subscheme $\{0=t_1=...=t_r\}$ in $\mathbb P^{r,w}_B$.\\
Put $E_{\mathbb W}:=\tau^{-1}(\mathbb W)$ and
$E_{\mathbf 0}:=\sigma^{-1}(\mathbf 0)$.\\
Note that $E_{\mathbb W}=\P^{r-1}\times_{B} \mathbb W$; \\
Write $\hat X\xleftarrow{i} E_{\mathbb W}=\P^{r-1}\times_B \mathbb W$  for the closed embedding.\\
Write $\dot X$ for $X_{s_1}$.\\
Let $Z\subset X$ be a closed subset with $Z\subset X-\mathbb W$.
Put $\dot Z=Z\cap \dot X$.
\end{notation}

\begin{notation}\label{j=tau -1_DVR}
Write $j: \dot X\hra \hat X_{x_1}$ for $\dot X\xrightarrow{\tau^{-1}} \hat X_{x_1}-(E_{\mathbb W})_{x_1}\hra \hat X_{x_1}$.\\
Clearly, $j|_{\dot Z}$ identifies $\dot Z$ with $\hat Z_{x_1}$; \\
Put $\dot q=\hat q\circ j: \dot X\to \mathbb A^{r-1}_B$; clearly, $\dot q=(s_2/s^{d_{1,2}}_1,...,s_r/s^{d_{1,r}}_1)$; \\
For an open $S$ in $\mathbb A^{r-1}_B$ put \\
$\dot X_S=\dot q^{-1}(S)$, $\hat X_S=\hat q^{-1}(S)$; $\dot Z_S=Z\cap \dot X_S$; $\hat Z_S=\hat Z\cap \hat X_S$; \\
clearly, $\dot X_S=j^{-1}(\hat X_S)$; \\
put $\dot q_{S}=\dot q|_{\dot X_S}: \dot X_S\to S$ and $\hat q_S=\hat q|_{\hat X_S}: \hat X_S\to S$ and $j_S=j|_{\dot X_S}: \dot X_S\hra \hat X_S$; \\
clearly, $j_S|_{\dot Z_S}$ identifies $\dot Z_S$ with $\hat Z_S$; \\
clearly, $\dot q_S=\hat q_S\circ j_S: \dot X_S\to S$; \\
put $\hat X_S\leftarrow S\times_{B} \mathbb W: \ i_S=i\times_{V} \mathbb W$; \\
clearly, $\hat q_S\circ i_S=pr_S: S\times_{B} \mathbb W \to S$.
\end{notation}

\begin{thm}\label{el_fib_DVR_fin_k_v}
Let $X\subset \Pro^n_B$ be a $B$-smooth closed subscheme, irreducible and of relative dimension $r\geq 1$,
${\bf x}\subset X$ be a finite set of closed points. 
Let $Z\subset X$ be a closed subset such that $\dim Z_b\leq r-1$.
Then under Notation of this section (particularly, Notation \ref{tau_and_sigma_DVR} and \ref{j=tau -1_DVR})
there exists a Zariski open neighborhood $S$ of the point $(0,b)\in \mathbb A^{r-1}_B$
such that the ${\bf x}$ is in $\dot X_S$ and the commutative diagram
\begin{equation}
\label{SquareDiagram_1}
    \xymatrix{
     \dot X_S\ar[drr]_{\dot q_S}\ar[rr]^{j_S}&&
\hat X_S\ar[d]^{\hat q_S}&&S\times_{B} \mathbb W \ar[ll]_{i_S}\ar[lld]^{pr_S} &\\
     && S  &\\    }
\end{equation}
is a diagram of the weak elementary $\dot Z_S$-fibration $\dot q_S: \dot X_S\to S$.
%
\end{thm}
Clearly, Theorem \ref{el_fib_DVR_fin_k_v} yields the following
\begin{cor}\label{cor: w_el_fib_inf_DVR}[= Theorem \ref{cor: w_el_fib_fin_A}]
Under the hypotheses and notation of Theorem \ref{el_fib_DVR_fin_k_v}
let $S$ be the open neighborhood in $\mathbb A^{r-1}_B$ of the point
$(0,b)\in \mathbb A^{r-1}_B$
as in Theorem \ref{el_fib_DVR_fin_k_v}.
Then the morphism
$\dot q_S: \dot X_S\to S$
is a weak elementary $\dot Z_S$-fibration.
%
\end{cor}

\begin{proof}[Proof of Theorem \ref{el_fib_DVR_fin_k_v}]
We begin with the following result which can be proved following arguments as in Section \ref{blowups}
(Proposition \ref{smooth2} and Lemma \ref{hat_q_fl_proj}).
\begin{prop}\label{smooth2DVR}
The scheme $\hat X_{x_1}$ is $B$-smooth.
The morphism $\hat q: \hat X_{x_1}\to \mathbb A^{r-1}_B$ is flat projective.
\end{prop}
Now using Notation \ref{j=tau -1_DVR} and
arguing as in Section \ref{blowups} construct a commutative diagram of the form
in the category of $B$-schemes
\begin{equation}
\label{SquareDiagramDVR}
    \xymatrix{
     \dot X\ar[drr]_{\dot q}\ar[rr]^{j}&&
\hat X_{x_1}\ar[d]_{\hat q}&&\mathbb A^{r-1}_{\mathbb W}\ar[ll]_{i}\ar[lld]^{\mathrm{pr}_{\mathbb A^{r-1}_B}} &\\
     && \mathbb A^{r-1}_B  &\\    }
\end{equation}
The $B$-scheme $\hat X_{x_1}$ is $B$-smooth.
The morphism $\hat q$ is flat projective. The morphism $\mathrm{pr}_{\mathbb A^{r-1}_B}$ is finite \'{e}tale.
The morphism
$i$ is a closed embedding, the morphism $j$ is an open embedding identifying $\dot X$ with
$\hat X_{x_1}-i(\mathbb A^{r-1}_i(\mathbb W))$. The ${\bf x}$ is contained in $\dot X$.
For the above closed subset $Z\subset X$ one has $\dot Z\subset \dot X$.
Prove that $\dot Z$ is finite over $\mathbb A^{r-1}_B$.

By the hypotheses of Theorem \ref{el_fib_DVR_fin_k_v} one has $\mathbb W_b\cap Z_b=\emptyset$.
Thus $\mathbb W\cap Z=\emptyset$.
It follows that the morphism $[s_1:...:s_r]: Z\to \mathbb P^{r-1,w}_B$ is well-defined.
It is affine and projective. Hence it is finite. This shows that
$\dot Z$ is finite over $\mathbb A^{r-1}_B$.
The following lemma is a straight forward consequence of Proposition \ref{smooth_fibre}.
\begin{lem}\label{smooth_fibreDVR}
Consider the scheme $C={\hat q}^{-1}((0,b))$. Then \\
1) $C$ is a $k(b)$-smooth projective curve, which is a closed subscheme of $\hat X_b$; \\
2) $\tau_b|_C: C\to X_b$ is a closed embedding; \\
3) $\tau_b(C)$ coincides with the smooth closed dimension one subscheme $(X_b)_{2,...,r}$ of $X_b$.
\end{lem}
By Proposition \ref{smooth2DVR} the $B$-scheme $\hat X_{x_1}$ is $B$-smooth and
the $\hat q: \hat X_{x_1}\to \mathbb A^{r-1}_B$ is flat projective.
Now Lemma \ref{smooth_fibreDVR} shows that
there exists a Zariski open neighborhood $S$ of the point $(0,b)\in \mathbb A^{r-1}_B$
such that for $\hat X_S:=\hat q^{-1}(S)$ the morphism $\hat q: \hat X_S\to S$ is smooth projective.
Take the base change of the diagram \eqref{SquareDiagramDVR} by means of the open embedding $S\hra \mathbb A^{r-1}_B$
\begin{equation}
\label{SquareDiagram2DVR}
    \xymatrix{
     \dot X_S\ar[drr]_{\dot q_S}\ar[rr]^{j_S}&&
\hat X_S\ar[d]_{\hat q_S}&&S\times_B {\mathbb W}\ar[ll]_{i_S}\ar[lld]^{pr_S} &\\
     && S  &\\    }
\end{equation}
This diagram shows that the morphism $\dot q_S: \dot X_S\to S$ is a weak elementary fibration. By Lemma \ref{smooth_fibreDVR} one has inclusions
${\bf x}\subset \dot q^{-1}_S((0,b)))\subset \dot X_S$. By Corollary \ref{cor: w_el_fib_inf_DVR}
for $\dot Z_S=Z\cap \dot X_S$
the morphism
$\dot q_S|_{\dot Z_S}: \dot Z_S\to S$ is finite.
Hence the morphism $\dot q_S: \dot X_S\to S$ is a weak elementary $\dot Z_S$-fibration.
Theorem \ref{el_fib_DVR_fin_k_v} is proved.
\end{proof}

\section{The case of an infinite field $k(b)$ }\label{sect_DVR-second}
The main aim of this section is to state Theorem \ref{thm_Artin_2_A}.\\
$A$ is a regular local ring, $m\subseteq A$ is its maximal ideal;\\
$K$ is the field of fractions of $A$;
$B=\mathrm{Spec}(A)$ and $b\in B$ is its closed point;\\
$k(b)$ is the residue field $A/m$; {\bf it is infinite} in this section;\\
$p>0$ is the characteristic of the field $k(b)$; \\
it is supposed in this notes that the field $K$ has characteritic zero;\\
$m\geq 1$ is an integer;\\
For an $B$-scheme $Y$ we write $Y_b$ for the closed fibre of the structure morphism $Y\to B$.\\
For an $B$-scheme morphism $q: Y'\to Y$ write $q_b: Y'_b\to Y_b$ for the morphism $q\times_B b$;\\
All schemes in this section are $B$-schemes and we write $Y\times Y'$ for $Y\times_B Y'$.
\begin{notation}\label{sect:2_DVR_2}
Let $\mathbb P^n:=\mathbb P^n_{B}$ be the projective space over $B$. \\
Let $S_A = A[T_0,...,T_n]$ be the homogeneous coordinate ring of $\mathbb P^n_B$.\\
Let $S_{A,d}\subset S_A$ be the $A$-submodule of homogeneous polynomials of degree $d$.\\
Let $F\in S_{A,d}$ be a degree $d$ homogeneous polynomial. \\
Let $H_{F}\subset \mathbb P^n_B$ be the closed subscheme of $\mathbb P^n_B$ defined by the homogeneous ideal $(F)\subset S_A$.\\
Let $S_{k(b)} = k(b)[T_0,...,T_n]$ be the homogeneous coordinate ring of $\mathbb P^n_b$.\\
Let $S_{k(b),d}\subset S_{k(b)}$ be the $k(b)$-subspace of homogeneous polynomials of degree $d$.\\
Let $X\subset \mathbb P^n_B$ be a $B$-smooth closed subscheme, irreducible and of relative dimension $r$.\\
Let ${\bf x}\subset X$ be a finite set of closed points.\\
Let $Z\subset X$ be a closed subset with $\dim Z_b < r$.\\
For each $F\in S_{A,d}$ put $X_F=X-H_{F}$ in $X$; \\\\
Let $d\gg 0$,
{\bf the generic} dimension $r$ vector $k(b)$-subspace $W$ in $S_{k(b),d}$
and
{\bf the generic} free $k(b)$-basis
${\bf f}=\{f_1,\dots,f_{r}\}$ of $W$
be as in Theorem \ref{GeneralSection_bar k} with respect to \\
the subscheme $X_b\subset \mathbb P^n_b$, its finite closed subset ${\bf x}$ and the closed subset $Z_b\subset X_b$. \\
For each the $f_i\in {\bf f}$ write $F_i\in S_{A,d}$ for a lift of $f_i$ to an element of $S_{A,d}$.\\
Let $X_i$ be the closed subscheme of $X$ equals $H_{f_i}\cap X$ (the scheme intersection). \\
For any $I\subset \{0,1,...,r\}$ write $X_I$ for the scheme intesection $\cap_{i\in I}X_i$ in $X$.\\
Write $\mathbb W=X_{1,...,r}$ for the closed subscheme in $X$.\\
The closed subscheme $\mathbb W_b\subset X_b$ depends only of $W$. \\
Write $\mathcal O(d)$ for $\mathcal O_{\mathbb P^n}(d)|_X$.\\
For the $f_i$ write $s_i\in \Gamma(X,\mathcal O(d))$ for the $F_i|_X$.\\
Put $X_{s_i}=X_{F_i}$ and $\dot X=X_{s_1}$.\\
Let $[x_1:...:x_r]$ be the homogeneous coordinates on $\mathbb P^{r-1}=\mathbb P^{r-1}_B$.
Put $\mathbb P^{r-1}_{x_i}=(\mathbb P^{r-1}_B)_{x_i}$.
Clearly, $\mathbb P^{r-1}_{x_1}\cong \mathbb A^{r-1}_B$.
Just identify
$[x_1:x_2:...:x_r]$ with $(x_2/x_1,...,x_r/x_1)$.\\
For the $W\subset S_{k(b),d}$ above and its the free $k(b)$-basis ${\bf f}=\{f_1,...,f_r\}$ put \\
$\dot q=(s_2/s_1,...,s_r/s_1): \dot X\to \mathbb A^{r-1}_B$.\\
For a Zariski open $S$ in $\mathbb A^{r-1}_B$ put $\dot X_S=\dot q^{-1}(S)$ and $\dot Z_S=Z\cap \dot X_S$.\\
Put $\dot q_S=\dot q|_{\dot X_S}: \dot X_S\to S$.
\end{notation}
The following partial result will be proven at the end of in Section \ref{blowups_A_inf_A}.
\begin{thm}\label{thm_Artin_2_A}
Under Notation \ref{sect:2_DVR_2} and the hypotheses of Theorem \ref{GeneralSection_bar k}
for each integer $d\gg 0$,
for {\bf the generic} dimension $r$ vector $k(b)$-subspace $W$ in $S_{k(b),d}$
and {\bf the generic} free $k(b)$-basis
${\bf f}=\{f_1,\dots,f_{r}\}$ of $W$
the following is true. \\
For the affine variety $\dot X$ and the morphism
$\dot q: \dot X\to \mathbb A^{r-1}_B$
there exists a Zariski open neighborhood $S\subset \mathbb A^{r-1}_B$ of the
finite set $\dot q({\bf x})$
such that the morphism \\
$\dot q_S: \dot X_S\to S$ is a weak elementary $\dot Z_S$-fibration.
\end{thm}

\begin{prop}\label{non_weighted_A}
Under the Notation \ref{sect:2_DVR_2}
the morphism
\begin{equation}\label{w_A}
\pi=[s_0:s_1:...:s_r]: X\to \mathbb P^{r}_B
\end{equation}
is well-defined and finite.
\end{prop}

\begin{proof}
One has $X=\cup^r_{i=0}X_{s_i}$, $\mathbb P^{r}=\cup^r_{i=0}\mathbb P^{r}_{t_i}$ and
$\pi^{-1}(\mathbb P^{r}_{t_i})=X_{s_i}$.
Since each $X_{s_i}$ is affine, the morphism $\pi$ is affine.
At the same time $\pi$ is projective. Thus, $\pi$ is finite.
\end{proof}

\begin{prop}\label{non_weighted2_A}
Under the hypotheses of Theorem \ref{thm_Artin_2_A} and Notation \ref{sect:2_DVR_2} let $M\subset X$ be a closed subscheme
such that $\mathbb W\cap M=\emptyset$. Then the morphism $[s_1:...:s_r]: M\to \mathbb P^{r-1}_B$
is finite.
Particularly, the morphism
$[s_1:...:s_r]: Z\to \mathbb P^{r-1}_B$
is finite.
\end{prop}

\begin{proof}
It is affine and projective.
\end{proof}

\section{Proof of the infinite $k(b)$-case of Theorem \ref{Very_General}}\label{blowups_A_inf_A}
The main aim of this section is to prove Theorem \ref{thm_Artin_2_A}.
We follow in this section Notation used in Section
\ref{sect_DVR-second}.

\begin{constr}\label{hat X_DVR_2}
For the $X\subset \P^n_B$, the $W$ and the ${\bf f}=\{f_1,...,f_r\}$ as above
suppose the closed subscheme $\mathbb W_b\subset X_b$
is finite \'{e}tale over $\mathrm{Spec}(k(b))$.
In this case the closed subscheme $\mathbb W\subset X$ is finite \'{e}tale over $B$.
Define $\hat X$ as a closed subscheme of $X\times \mathbb P^{r-1}_B$
given by equations $s_jx_i=s_ix_j$.
We regard $\hat X$ as the blowup of $X$ at the subscheme $\mathbb W$.

\end{constr}
The projection $X\times \mathbb P^{r-1}\to \mathbb P^{r-1}$
and $\mathbb P^{r}\times \mathbb P^{r-1}\to \mathbb P^{r-1}$
induces a morphism
$\hat q: \hat X\to \mathbb P^{r-1}_B$
and
$\hat p: \hat {\mathbb P}^{r}_B\to \mathbb P^{r-1}_B$
allowing to consider the scheme $\hat X$
and $\hat {\mathbb P}^{r}_b$
as a $\mathbb P^{r-1}_B$-scheme.

\begin{notation}\label{tau_and_sigma_DVR_2}
The projection $X\times \mathbb P^{r-1}\to X$
and $\mathbb P^{r}\times \mathbb P^{r-1}\to \mathbb P^{r}$
induces a morphism
$\tau: \hat X\to X$
and
$\sigma: \hat {\mathbb P}^{r}\to \mathbb P^{r}$ respectively.\\
Write
$\mathbf 0$ for $[1:0:...:0]_B$ in $\mathbb P^{r}_B=\mathbb P^{r}$.\\
Put $E_{\mathbb W}=\tau^{-1}(\mathbb W)$
and
$E_0=\sigma^{-1}(\mathbf 0)$.\\
Note that $E_{\mathbb W}=\P^{r-1}_B\times_{B} \mathbb W=\P^{r-1}_{\mathbb W}$; \\
Write $\hat X\xleftarrow{i} E_{\mathbb W}=\P^{r-1}_{\mathbb W}$ \ for the closed embedding.\\
Write as above $\dot X$ for $X_{s_1}$.\\
Let $Z\subset X$ be the closed subset as in Notation \ref{sect:2_DVR_2}; then $Z\subset X-\mathbb W$.\\
Put $\dot Z=Z\cap \dot X$ as in Notation \ref{sect:2_DVR_2}.
\end{notation}

\begin{rem}
Clearly, the morphism
$\tau: \hat X-E_{\mathbb W}\to X-\mathbb W$
is an isomorphisms. \\
Put $\hat Z$ for $\tau^{-1}(Z)$ and $Z_{x_1}=\hat Z\cap \hat X_{x_1}$.
\end{rem}

\begin{prop}\label{X_s_1_DVR_2}
Put $\tau_1=\tau|_{\hat X_{x_1}-(E_{\mathbb W})_{x_1}}: \hat X_{x_1}-(E_{\mathbb W})_{x_1}\to \dot X$. Then \\
1) the morphism
$\tau_1: \hat X_{x_1}-(E_{\mathbb W})_{x_1}\to \dot X$
is an isomorphism; \\
2) the schemes $\hat X_{x_1}-(E_{\mathbb W})_{x_1}$ and
$\dot X$
are $B$-smooth.\\
3) the morphism
$\tau_1|_{\hat Z_{x_1}}: \hat Z_{x_1}\to \dot Z$
is an isomorphism.
\end{prop}

\begin{notation}\label{j=tau -1_DVR_2}
Write $j: \dot X\hra \hat X_{x_1}$ for $\dot X\xrightarrow{\tau^{-1}} \hat X_{x_1}-(E_{\mathbb W})_{x_1}\hra \hat X_{x_1}$.\\
Clearly, $j|_{\dot Z}$ identifies $\dot Z$ with $\hat Z_{x_1}$; \\
Clearly, $\dot q=\hat q\circ j: \dot X\to \mathbb A^{r-1}_B$; where $\dot q=(s_2/s_1,...,s_r/s_1)$ as in Notation \ref{sect:2_DVR_2}; \\
For an open $S$ in $\mathbb A^{r-1}_B$ put
$\hat X_S=\hat q^{-1}(S)$ and
$\hat Z_S=\hat Z\cap \hat X_S$; \\
Following Notation \ref{sect:2_DVR_2} one has $\dot X_S=\dot q^{-1}(S)$ and $\dot Z_S=Z\cap \dot X_S$;\\
clearly, $\dot X_S=j^{-1}(\hat X_S)$; \\
put $\dot q_{S}=\dot q|_{\dot X_S}: \dot X_S\to S$ and $\hat q_S=\hat q|_{\hat X_S}: \hat X_S\to S$ and $j_S=j|_{\dot X_S}: \dot X_S\hra \hat X_S$; \\
clearly, $j_S|_{\dot Z_S}$ identifies $\dot Z_S$ with $\hat Z_S$; \\
clearly, $\dot q_S=\hat q_S\circ j_S: \dot X_S\to S$; \\
put $\hat X_S\leftarrow S\times_{B} \mathbb W: \ i_S=i\times_{B} \mathbb W$; \\
clearly, $\hat q_S\circ i_S=pr_S: S\times_{B} \mathbb W \to S$.
\end{notation}

\begin{thm}\label{el_fib_DVR_inf_k_v}
Let $X\subset \Pro^n_B$ a $B$-smooth closed subscheme, irreducible and of relative dimension $r$,
${\bf x}\subset X$ be a finite subset of closed point.
Let $d\gg 0$,
{\bf $W$ in $S_{k(b),d}$ be the generic} dimension $r$ vector subspace
and
${\bf f}=\{f_1,\dots,f_{r}\}$ be its {\bf the generic} free $k(b)$-basis
as in of Corollary \ref{GeneralSection_inf k}.
Then there exists a neighborhood $S$ in $\mathbb A^{r-1}_B$ of the finite set $\dot q({\bf x})$
such that under Notation \ref{sect:2_DVR_2}, \ref{tau_and_sigma_DVR_2} and \ref{j=tau -1_DVR_2} the commutative diagram
of $B$-schemes and $B$-morphisms
\begin{equation}
\label{SquareDiagram_0_DVR_2}
    \xymatrix{
     \dot X_S\ar[drr]_{\dot q_S}\ar[rr]^{j_S}&&
\hat X_S\ar[d]^{\hat q_S}&&S\times_{B} \mathbb W\ar[ll]_{i_S}\ar[lld]^{pr_S} &\\
     && S  &\\    }
\end{equation}
is a diagram of the weak elementary fibration $\dot q_S: \dot X_S\to S$.

Let $Z\subset X$ be a closed subset with $\dim Z_b < r$.
Then one can choose the $W$ and the ${\bf f}$ above such that
the properties (i) to (iii) as in Corollary \ref{GeneralSection_inf k}
does hold and
$\mathbb W_b\cap Z_b=\emptyset$.
In this case the diagram \eqref{SquareDiagram_0_DVR_2} is a diagram of
the weak elementary $\dot Z_S$-fibration $\dot q_S: \dot X_S\to S$.
\end{thm}
Clearly, Theorem \ref{el_fib_DVR_inf_k_v} yields the following
\begin{cor}\label{cor: w_el_fib_inf_field_DVR_2}[= Theorem \ref{thm_Artin_2_A}]
Under the hypotheses of Theorem \ref{el_fib_DVR_inf_k_v}
let $S$ be the open neighborhood in $\mathbb A^{r-1}_B$ of the finite set
$\dot q({\bf x})$
as in Theorem \ref{el_fib_DVR_inf_k_v}.
Then the morphism \\
$\dot q_S: \dot X_S\to S$
is a weak elementary fibration.

Let $Z\subset X$ be a closed subset with $\dim Z_v < r$. Then one can choose
the $W$ and the ${\bf f}$ above subject to conditions
(i) to (iii) as in Corollary \ref{GeneralSection_inf k}
and additionally subject to the condition
$\mathbb W_b\cap Z_b=\emptyset$.
In this case the morphism
$\dot q_S: \dot X_S\to S$
is a weak elementary $\dot Z_S$-fibration.
\end{cor}

\begin{proof}[Proof of Theorem \ref{el_fib_DVR_inf_k_v}]
The projections $X\times \mathbb P^{r-1}\to \mathbb P^{r-1}$
and
$\mathbb P^{r}\times \mathbb P^{r-1}\to \mathbb P^{r-1}$
induce morphisms
$\hat q: \hat X\to \mathbb P^{r-1}$
and
$\hat p: \hat {\mathbb P}^{r}\to \mathbb P^{r-1}$
allowing to consider the scheme $\hat X$
and $\hat {\mathbb P}^{r}$
as a $\mathbb P^{r-1}$-scheme.
Let $[x_1:...:x_r]$ be homogeneous coordinates on $\mathbb P^{r-1}$
and consider its the open subscheme
$\mathbb P^{r-1}_{x_1}$.
The morphism
$\pi\times id: X\times \mathbb P^{r-1}\to \mathbb P^{r}\times \mathbb P^{r-1}$
induces a morphism
$$\hat \pi: \hat X\to \hat {\mathbb P}^{r}_B.$$
This is a morphism of the $\mathbb P^{r-1}_B$-schemes. Put
$\hat X_{x_1}=q^{-1}(\mathbb P^{r-1}_{x_1})$
and
$\hat {\mathbb P}^{r}_{x_1}=p^{-1}(\mathbb P^{r-1}_{x_1})$.
Recall that $\mathbb P^{r-1}_{x_1}$ is the affine space $\mathbb A^{r-1}_B$.

Take $(e_0,e_1,...,e_r)$ from Section \ref{blowups}
equals $(1,1,...,1)$.
Now repeating literally arguments proving
Lemma \ref{hat_q_fl_proj} we see that
the morphism
$\hat q: \hat X_{x_1}\to \hat {\mathbb P}^{r-1}_{x_1}=\mathbb A^{r-1}_B$
is flat projective.
In turn,
repeating literally arguments proving
Proposition \ref{smooth2} we see that
the scheme $\hat X_{x_1}$ is $B$-smooth.

By the item (iii) of Corollary \ref{GeneralSection_inf k}
the morphism $\hat q_b: \hat X_b \to \P^{r-1}_b$ is smooth projective of relative dimension 1
over a neighborhood in $\mathbb A^{r-1}_b$ of the finite set $\dot q_b({\bf x})$.
Thus, the morphism
$\hat q_S: \hat X_S \to S$ is smooth projective for a
a neighborhood in $\mathbb A^{r-1}_B$ of the finite set $\dot q({\bf x})$.
Theorem \ref{el_fib_DVR_inf_k_v} is proved.\qedhere
\end{proof}

\end{document}